\documentclass[12pt]{amsart}
\usepackage{a4wide,enumerate,color,graphicx}
\usepackage{amsmath,bm}
\usepackage{scrextend} 
\allowdisplaybreaks
\usepackage[normalem]{ulem}
\usepackage[dvipsnames]{xcolor}

\let\pa\partial  
\let\na\nabla  
\let\eps\varepsilon

\newcommand{\R}{{\mathbb R}} 
\newcommand{\diver}{\operatorname{div}}

\newcommand{\ran}{\operatorname{ran}}

\newtheorem{theorem}{Theorem}
\newtheorem{lemma}[theorem]{Lemma}

\newtheorem{remark}[theorem]{Remark}
\newtheorem{corollary}[theorem]{Corollary}

\begin{document}  

\title[Weak-strong uniqueness Maxwell-Stefan]{Weak-strong uniqueness for
Maxwell--Stefan systems}

\author[X. Huo]{Xiaokai Huo}
\address{Institute for Analysis and Scientific Computing, Vienna University of  
	Technology, Wiedner Hauptstra\ss e 8--10, 1040 Wien, Austria}
\email{xiaokai.huo@tuwien.ac.at} 

\author[A. J\"ungel]{Ansgar J\"ungel}
\address{Institute for Analysis and Scientific Computing, Vienna University of  
	Technology, Wiedner Hauptstra\ss e 8--10, 1040 Wien, Austria}
\email{juengel@tuwien.ac.at} 

\author[A. Tzavaras]{Athanasios E. Tzavaras}
\address{Computer, Electrical and Mathematical Science and Engineering Division,
King Abdullah University of Science and Technology (KAUST),
Thuwal 23955-6900, Saudi Arabia}
\email{athanasios.tzavaras@kaust.edu.sa}

\date{\today}

\thanks{The second author has been partially supported by the Austrian Science Fund 
(FWF), grants P30000, P33010, F65, and W1245. 
This work received funding from the European 
Research Council (ERC) under the European Union's Horizon 2020 research and 
innovation programme, ERC Advanced Grant NEUROMORPH, no.~101018153.} 

\begin{abstract}
The weak-strong uniqueness for Maxwell--Stefan systems and 
some generalized systems is proved. 
The corresponding parabolic cross-diffusion equations are considered in a bounded
domain with no-flux boundary conditions. 
The key points of the proofs are various inequalities for the relative entropy
associated to the systems and the analysis of the spectrum of a quadratic form
capturing the frictional dissipation. The latter task is complicated by the
singular nature of the diffusion matrix. This difficulty is addressed by
proving its positive definiteness on a subspace and using the Bott--Duffin
matrix inverse. The generalized Maxwell--Stefan systems are shown to cover
several known cross-diffusion systems for the description of tumor growth and
physical vapor deposition processes.
\end{abstract}

\keywords{Cross diffusion, weak-strong uniqueness, relative entropy, Maxwell--Stefan
systems, thin-film solar cell model, tumor-growth model.}  
 
\subjclass[2000]{35A02, 35K51, 35K55, 35Q35.}

\maketitle


\section{Introduction}

The Maxwell--Stefan equations describe the diffusive transport of the components
of gaseous mixtures. Applications arise in, e.g., sedimentation, dialysis, 
electrolysis, and ion exchange \cite{WeKr00}. 
They were suggested in 1866 by James Maxwell \cite{Max66} for dilute gases
and in 1871 by Josef Stefan \cite{Ste71} for fluids. 
While there are several works on the existence of 
local-in-time smooth solutions \cite{Bot11,GoMa98,HMPW17} 
and global-in-time weak solutions \cite{JuSt13} in the case of vanishing barycentric
velocity, the problem of the uniqueness of solutions is basically unsolved.
The uniqueness of strong solutions has been shown in \cite{HMPW17,HuSa18},
and uniqueness results for weak solutions in a very special case can be found
in \cite{ChJu18}. In this paper, we make a step forward in the uniqueness problem
by showing that strong solutions are unique in the class of weak solutions
to Maxwell--Stefan systems.

\subsection{Setting}

We consider an ideal gaseous mixture consisting of $n$ components with
volume fractions or concentrations $c_i(x,t)$, $i=1,\ldots,n$. The dynamics
of the mixture is given by the mass balance equations and
the relations between the driving forces and the fluxes,
\begin{equation}\label{1.eq}
  \pa_t c_i + \diver(c_iu_i) = 0, \quad \na c_i = -\sum_{j=1}^n
	\frac{c_ic_j}{D_{ij}}(u_i-u_j), \quad
	i=1,\ldots,n,
\end{equation}
where $u_i(x,t)$ are the partial velocities and $D_{ij}=D_{ji}>0$ are 
diffusion coefficients. 
The equations are solved in a bounded domain $\Omega\subset\R^d$ ($d\ge 1$), 
supplemented by the initial and no-flux boundary conditions
\begin{equation}\label{1.bic}
  c_i(0) = c_i^0\quad\mbox{in }\Omega, \quad
	\na c_i\cdot\nu = 0\quad\mbox{on }\pa\Omega,\ t>0,\ i=1,\ldots,n,
\end{equation}
where $\nu$ is the exterior unit normal vector to $\pa\Omega$. 

We assume that the barycentric velocity vanishes, which implies
that the sum of all fluxes vanishes, $\sum_{i=1}^n c_iu_i=0$. Then, supposing 
that $c_i^0\ge 0$ and $\sum_{i=1}^n c_i^0=1$ in $\Omega$, we deduce from
mass conservation that
\begin{equation*}
\sum_{i=1}^n c_i = 1 \quad \mbox{in $\Omega$ for all $t>0$.}
\end{equation*}
This constraint is
necessary to invert the force-flux relations in \eqref{1.eq}, i.e.\ to express
the flux $c_iu_i$ as a linear combination of the driving forces $\na c_j$.

The global existence analysis for \eqref{1.eq}--\eqref{1.bic} is based on the
property that the system is endowed with the entropy functional
\begin{equation}\label{1.ent}
  H(\bm{c}) = \sum_{i=1}^n\int_\Omega c_i(\log c_i-1)dx,
\end{equation}
where $\bm{c}=(c_1,\ldots,c_n)$ solves \eqref{1.eq}--\eqref{1.bic} and satisfies
the entropy dissipation inequality
\cite[(1.14)]{JuSt13}
\begin{equation}\label{1.ei}
  \frac{dH}{dt}(\bm{c}) + C \sum_{i=1}^n \int_\Omega  |\na\sqrt{c_i}|^2dx \le 0,
\end{equation}
with $C>0$ depending only on $(D_{ij})$. The aim of this paper is to prove the
weak-strong uniqueness for \eqref{1.eq}--\eqref{1.bic} and generalized systems.
Weak-strong uniqueness means that any weak solution coincides with a strong 
solution emanating from the same initial data as long as the latter exists.
In other words, the strong solutions must be unique within the class of weak solutions.
To achieve this aim, we use ideas from our previous work \cite{HJT19} and
establish a relative entropy inequality. This leads to a stability estimate for
the difference of a weak and a strong solution and eventually to the
weak-strong uniqueness property. Here, the relative entropy functional is given by
\begin{equation}\label{1.relent}
  H(\bm{c}|\bar{\bm{c}}) = \sum_{i=1}^n\int_\Omega 
	\bigg(c_i\log\frac{c_i}{\bar{c}_i}-(c_i-\bar{c}_i)\bigg)dx,
\end{equation}
where $\bm{c}$ and $\bar{\bm{c}}$ are suitable solutions to \eqref{1.eq}--\eqref{1.bic}.

In the literature, relative entropies are known to be useful to prove the weak-strong
uniqueness of solutions. First results were achieved for systems of hyperbolic
conservation laws \cite{Daf79} and later for the compressible Navier--Stokes
equations \cite{FJN12,FeNo12} and general hyperbolic-parabolic systems endowed 
with an entropy \cite{CT18}. 
The relative entropy technique was applied to,
for instance, entropy-dissipating reaction-diffusion equations \cite{Fis17},
reaction-cross-diffusion systems \cite{ChJu19}, 
energy-reaction-diffusion systems \cite{Hop21}, 
nonlocal cross-diffusion systems \cite{JPZ21},
and quantum Euler systems \cite{BGL19,GLT17}.
Compared to the results of, e.g.\ \cite{ChJu19,Hop21}, the diffusion matrix
in these works is assumed to be positive definite 
if $c_i>0$ for all $i=1,\ldots,n$, which is
not satisfied for the Maxwell--Stefan system.

\subsection{Definitions and assumptions}

We impose the following assumptions:
\renewcommand{\labelenumi}{(A\theenumi)}
\begin{enumerate}
\item Domain: $\Omega\subset\R^d$ with $d\ge 1$ is a bounded domain.
\item Coefficients: $D_{ij}>0$ and $D_{ij}=D_{ji}$ for all $i,j=1,\ldots,n$, $i\neq j$.
\item Initial data: $0\le c_{i}^0\in L^1(\Omega)$ for $i=1,\ldots,n$, 
$H(\bm{c}^0)<\infty$, and $\sum_{i=1}^n c_{i}^0=1$ in $\Omega$.
\end{enumerate}
Next, we define the concept of weak and strong solutions  employed in this paper.

We call $\bm{c} = (c_1,\ldots,c_n)$ a {\em weak solution} to \eqref{1.eq}--\eqref{1.bic}
if $\bm{c}$ satisfies the initial condition \eqref{1.bic}, 
$c_i\ge 0$, $\sum_{i=1}^n c_i=1$ in $\Omega\times(0,\infty)$, 
$$
  \sqrt{c_i}\in L^2_{\rm loc}(0,\infty;H^1(\Omega)), \quad 
	c_i\in C_{\rm loc}^0([0,\infty);{\mathcal V}'), \quad i=1,\ldots,n,
$$
where ${\mathcal V}'$ is the dual space of ${\mathcal V}=\{w\in H^2(\Omega):
\na w\cdot\nu=0$ on $\pa\Omega\}$, and $\bm{c}$ 
satisfies \eqref{1.eq}--\eqref{1.bic} in the weak sense, i.e., for any
$\phi_i\in C_{\rm loc}^1([0,\infty);C^1(\overline\Omega))$ satisfying
$\na\phi_i\cdot\nu=0$ on $\pa\Omega$ and any $T>0$, $i=1,\ldots,n$, we have
\begin{equation*}
  \int_\Omega c_i(T)\phi_i(T)dx - \int_\Omega c_i^0\phi_i(0)dx
	- \int_0^T\int_\Omega c_i\pa_t\phi_i dxdt
	- \int_0^T\int_\Omega c_iu_i\cdot\na\phi_i dxdt = 0,
\end{equation*}
where $u_i$ satisfies the force-flux relations in \eqref{1.eq}.
The last integral is well defined, since the gradient bound for $\sqrt{c_i}$
implies that $\sqrt{c_i}u_i\in L_{\rm loc}^2(0,\infty;L^2(\Omega))$ (see Lemma
\ref{lem.cu2} below) and thus, because of the property $0\le c_i\le 1$,
$c_iu_i\in L_{\rm loc}^2(0,\infty;L^2(\Omega))$. 
Finally, a weak solution is required to satisfy
the entropy inequality
\begin{equation}\label{1.entropy}
  H(\bm{c}(t)) + \frac12 \sum_{i,j=1}^n \int_0^t\int_\Omega
	\frac{c_ic_j}{D_{ij}}|u_i-u_j|^2 dxds \le H(\bm{c}^0).
\end{equation}
For the Maxwell-Stefan system this is not an additional requirement 
as it is guaranteed by the existence theory of \cite{JuSt13};
see Section \ref{sec.ex.ms}.

We will use the term {\em strong solution} to \eqref{1.eq}--\eqref{1.bic}
to mean that $\bar{\bm{c}}=(\bar c_i,\ldots,\bar c_n)$ 
with $0 < \bar{c}_i < 1$ is a weak solution satisfying additional 
regularity properties. The necessary regularity is 
stated precisely in context. In certain cases, $\bar c_i$ satisfies 
\eqref{1.eq}--\eqref{1.bic} pointwise, 
as is the traditional notion of strong solutions. 

\subsection{Main results and key ideas of the proofs}

Our first main result is concerned with the Maxwell--Stefan system
\eqref{1.eq}--\eqref{1.bic}.

\begin{theorem}[Weak-strong uniqueness]\label{thm.ws}
Let Assumptions (A1)--(A2) hold. Let $\bm{c}$ be a weak solution to
\eqref{1.eq}--\eqref{1.bic} and let $\bar{\bm{c}}$ be a strong solution
to \eqref{1.eq}--\eqref{1.bic} satisfying $0<\bar{c}_i<1$ in $\Omega$, $t>0$,
the regularity properties
$$
  \log\bar{c}_i\in H_{\rm loc}^1(\Omega\times(0,\infty)), \quad
	\bar{u}_i\in L_{\rm loc}^\infty(\Omega\times(0,\infty)),
$$
and $\bar{c}_i$ does not have anomalous dissipation, i.e., it satisfies
the entropy identity
$$
  H(\bar{\bm{c}}(t)) + \frac12\int_0^t\int_\Omega\sum_{i,j=1}^n
	\frac{\bar{c}_i\bar{c}_j}{D_{ij}}|\bar{u}_i-\bar{u}_j|^2 dxds = H(\bar{\bm{c}}^0)
	\quad\mbox{for }t>0.
$$
The initial data for $\bm{c}$ and $\bar{\bm{c}}$ satisfy Assumption (A3).
Then for any $t>0$, there exists a constant $C(t)>0$, depending on $t$, $\Omega$,
$n$, and $(D_{ij})$, such that
\begin{equation}\label{1.rei}
  H(\bm{c}(t)|\bar{\bm{c}}(t)) + \sum_{i=1}^n\int_0^t\int_\Omega c_i|u_i-\bar{u}_i|^2
	dxds \le C(t)H(\bm{c}^0|\bar{\bm{c}}^0).
\end{equation}
If the initial data coincide, i.e.\ $\bm{c}^0=\bar{\bm{c}}^0$ in $\Omega$, then
$\bm{c}(t)=\bar{\bm{c}}(t)$ in $\Omega$ for $t>0$.
\end{theorem}

We verify in Section \ref{sec.ex.ms} that solutions with the stated regularity exist.
To prove Theorem \ref{thm.ws}  we develop a relative entropy identity and use it
as a yardstick to control the distance between two solutions.
First, it is shown that the relative entropy \eqref{1.relent} satisfies the inequality 
\begin{align}\label{1.aux}
  \frac{dH}{dt}(\bm{c}|\bar{\bm{c}}) &+ \frac12\sum_{i,j=1,\,i\neq j}^n\int_\Omega
	\frac{c_ic_j}{D_{ij}}|(u_i-\bar{u}_i)-(u_j-\bar{u}_j)|^2 dx \\
	&\le  - \sum_{i,j=1,\,i\neq j}^n\int_\Omega\frac{c_i}{D_{ij}}(c_j-\bar{c}_j)
	(u_i-\bar{u}_i)\cdot(\bar{u}_i -\bar{u}_j)dx.
	\nonumber 
\end{align}
(see Section \ref{sec.relent}).
Next, we study how the frictional dissipation 
(the second term in \eqref{1.aux}) controls the $L^2$ norm of $u_i-\bar{u}_i$. 
The quadratic form in \eqref{1.aux} captures the dissipative effect of friction
in the following way:
\begin{align}
  \frac12\sum_{i,j=1,\,i\neq j}^n &
	\frac{c_ic_j}{D_{ij}}|(u_i-\bar{u}_i)-(u_j-\bar{u}_j)|^2
	= \sum_{i=1}^n c_i (u_i - \bar{u}_i) \cdot \sum_{j=1}^n \frac{1}{D_{i j}} 
	c_j \big((u_i-\bar{u}_i)-(u_j-\bar{u}_j)\big) \nonumber \\
	&= \sum_{i,j = 1}^n A_{ij}(c)(\sqrt{c_i} (u_i - \bar{u}_i)) \cdot 
	(\sqrt{c_j} (u_j - \bar{u}_j)) = \bm{Y}^T A(\bm{c})\bm{Y}, \label{1.fricdiss}
\end{align}
where the matrix 
$A(\bm{c})=(A_{ij}(\bm{c}))\in\R^{n\times n}$ is defined by
\begin{equation}\label{1.A}
  A_{ij}(\bm{c}) = \left\{\begin{array}{ll}
	\sum_{k=1,\,k\neq i}^n c_k/D_{ik} &\quad\mbox{if }i=j, \\
	-\sqrt{c_ic_j}/D_{ij} &\quad\mbox{if }i\neq j ,
	\end{array}\right.
\end{equation}
and $\bm{Y}=(Y_1,\ldots,Y_n)$ with $Y_i=\sqrt{c_i}(u_i-\bar{u}_i)$.
The matrix $A(\bm{c})$ is singular and thus not positive definite. However,
we can show that it is positive definite on the subspace
$L:=\{\bm{z}\in\R^n:\sqrt{\bm{c}}\cdot\bm{z}=0\}$ (here, $\sqrt{\bm{c}}$ is the
vector with components $\sqrt{c_i}$) and the quadratic form satisfies
$$
  \bm{Y}^T A(\bm{c})\bm{Y} \ge \mu|P_L\bm{Y}|^2,
$$
where $\mu>0$ is a uniform lower bound for the positive eigenvalues of $A(\bm{c})$
and $P_L$ is the projection on $L$. This inequality and a careful estimate
of the right-hand side of \eqref{1.aux} implies \eqref{1.rei} and the
weak-strong uniqueness property.

The $L^\infty$ bound on the partial velocities $\bar{u}_i$ in Theorem \ref{thm.ws}
can be avoided at the expense of assuming $\na\sqrt{c_i} \in L^\infty$ and 
$\bar{c}_i$ is uniformly bounded from below by a positive constant. 
The uniform lower bound is not needed in
Theorem \ref{thm.ws}, where only positivity is required.

\begin{corollary}\label{coro}
Let the assumptions of Theorem \ref{thm.ws} hold, replacing $\bar{u}_i \in L^\infty (\Omega\times(0,\infty))$ by $\nabla \sqrt{\bar{c}_i}\in L^\infty (\Omega\times(0,\infty))$, 
$i=1,\ldots,n$. 
Suppose additionally that there exists $m>0$ such that $\bar{c}_i(t)\ge m$ in
$\Omega$, $t>0$, $i=1,\ldots,n$. Then there exist constants $C_1>0$
and $C_2(t)>0$ (depending on $t$, $\Omega$, $n$, and $(D_{ij})$) 
such that the following inequality holds for $t>0$:
\begin{equation}\label{1.rei2}
  H(\bm{c}(t)|\bar{\bm{c}}(t)) + C_1\sum_{i=1}^n\int_0^t\int_\Omega
	|\na(\sqrt{c_i}-\sqrt{\bar{c}_i})|^2 dxds \le C_2(t) H(\bm{c}^0|\bar{\bm{c}}^0).
\end{equation}
\end{corollary}

The relative entropy inequality \eqref{1.rei2} is the analogue of the
entropy estimate \eqref{1.ei}. It can be achieved by working with the square 
roots $\sqrt{c_i}$ as the main variables. More precisely, we write the force-flux
relations in \eqref{1.eq} as
\begin{equation}\label{1.eqA}
  2\na\sqrt{c_i} = -\sum_{j=1}^n A_{ij}(\bm{c})\sqrt{c_j}u_j, \quad i=1,\ldots,n,
\end{equation}
subject to $\sum_{i=1}^n c_iu_i=0$, where $A(\bm{c})$ is defined in \eqref{1.A}. 
This system cannot be directly inverted,
since $\ker A(\bm{c}) =\operatorname{span}\{\sqrt{\bm{c}}\}$. However,
introducing the Bott--Duffin inverse $A^{BD}(\bm{c})$ of $A(\bm{c})$ with respect to 
$L:=(\operatorname{span}\{\sqrt{\bm{c}}\})^\perp$ (see Section \ref{sec.A} and Appendix \ref{app}), we can 
invert \eqref{1.eqA}, leading to
\begin{equation}\label{1.eqABD}
  \sqrt{c_i}u_i = -2\sum_{j=1}^n A_{ij}^{BD}(\bm{c})\na\sqrt{c_j}, \quad
	i=1,\ldots,n,
\end{equation}
and system \eqref{1.eq} can be formulated in the concise form
\begin{equation}\label{1.concise}
  \pa_t c_i = 2\diver\bigg(\sum_{j=1}^n\sqrt{c_i}A_{ij}^{BD}(\bm{c})\na\sqrt{c_j}
	\bigg), \quad i=1,\ldots,n.
\end{equation}
The Bott--Duffin inverse $A^{BD}(\bm{c})$ equals the group inverse
studied in \cite{BoDr20}, since $L=\ran A(\bm{c})$. Compared to \cite{BoDr20},
we work here with the square roots $\sqrt{c_i}$ instead of the chemical
potentials $\log c_i$ (see \cite[(4.25)]{BoDr20}).
The relative entropy inequality \eqref{1.aux} is rewritten in the form
(see Lemma \ref{lem.relent2})
\begin{align}\label{1.aux2}
  \frac{dH}{dt}(\bm{c}|\bar{\bm{c}}) &+ 4\sum_{i,j=1}^n\int_\Omega
	A_{ij}^{BD}(\bm{c})Z_i\cdot Z_j dx \\
	&\le 4\sum_{i,j=1}^n\int_\Omega
	Z_i\cdot\na\sqrt{\bar{c}_j}\bigg(\frac{\sqrt{c_i}}{\sqrt{\bar{c}_i}}
	A_{ij}^{BD}(\bar{\bm{c}})
	- A_{ij}^{BD}( \bm{c})\frac{\sqrt{c_j}}{\sqrt{\bar{c}_j}}\bigg)dx, \nonumber
\end{align} 
where $Z_i=\na\sqrt{c_i}-\sqrt{c_i/\bar{c}_i}\na\sqrt{\bar{c}_i}$, $i=1,\ldots,n$.
We prove in Lemma \ref{lem.A} that the Bott--Duffin inverse is symmetric and
positive definite on $L$,
$$
  \bm{Z}^T A^{BD}(\bm{c})\bm{Z}\ge \lambda|P_L\bm{Z}|^2,
	\quad \bm{Z}=(Z_1,\ldots,Z_n),
$$
where $\lambda>0$ is a uniform lower bound for the positive eigenvalues of
$A^{BD}(\bm{c})$. Inequality \eqref{1.rei2} now follows from this property and
suitable estimates for the right-hand side of \eqref{1.aux2}.

The above-mentioned techniques can be extended to a class of generalized
Maxwell--Stefan systems, which includes several examples of
cross-diffusion systems occurring in applications (see Section \ref{sec.exam}): 
\begin{align}
  \pa_t c_i + \diver(c_iu_i) &= 0, \quad \sum_{j=1}^n c_ju_j = 0, \label{1.gms1} \\
	-\sum_{j=1}^n K_{ij}(\bm{c}) c_ju_j &= c_i\na\frac{\delta H}{\delta c_i}(\bm{c})
	- c_i \sum_{j=1}^n c_j\na\frac{\delta H}{\delta c_j}(\bm{c}), \quad
	i=1,\ldots,n, \label{1.gms2}
\end{align}
together with the initial and boundary conditions \eqref{1.bic}, where
$\delta H/\delta c_i$ denotes the variational derivative of $H$. Again
$\sum_{i=1}^n c_{i}^0=1$ implies that $\sum_{i=1}^n c_i(t)=1$ in $\Omega$, 
$t>0$. We assume that
$$
  H(\bm{c}) = \sum_{i=1}^n \int_\Omega h_i(c_i)dx,
$$
which gives $\delta H/\delta c_i=h_i'$, and $(K_{ij})\in\R^{n\times n}$ 
satisfies $\sum_{i=1}^n K_{ij}(\bm{c})=0$ for all $\bm{c}\in\R_+^n$.
This model was proposed in \cite{HJT19} and can be obtained as the high-friction
limit of multicomponent Euler systems. It can also be derived from elementary
thermodynamic considerations; see Appendix \ref{app.thermo}.
If the entropy $H(\bm{c})$ equals
\eqref{1.ent} and $K_{ij}(\bm{c})=\sqrt{c_i}A_{ij}(\bm{c})/\sqrt{c_j}$,
where $A_{ij}(\bm{c})$ is defined in \eqref{1.A}, then system 
\eqref{1.gms1}--\eqref{1.gms2} reduces to \eqref{1.eq}. 
We refer to \cite{BoDr15,CuBy99,Lam06} for multicomponent diffusion models 
that account for other factors, such as thermal conduction, viscous stresses, 
chemical reactions, etc.

We introduce the matrix $B(\bm{c})=(B_{ij}(\bm{c}))\in\R^{n\times n}$ by
\begin{equation}\label{1.AK}
  B_{ij}(\bm{c}) = \frac{1}{\sqrt{c_i}}K_{ij}(\bm{c})\sqrt{c_j}, \quad
	i,j=1,\ldots,n,
\end{equation}
and we assume that $B(c)$ is symmetric
and as before, we set $L:=\{\bm{z}\in\R^n:\sqrt{\bm{c}}\cdot\bm{z}=0\}$ and 
$L^\perp = \mbox{span}\{\sqrt{\bm{c}}\}$.
We write \eqref{1.gms2} as (see the beginning of Section \ref{sec.gen})
$$
  -\sum_{j=1}^n B_{ij}(\bm{c})\sqrt{c_j}u_j
	= \sum_{j=1}^n(P_L)_{ij}\sqrt{c_j}\na h_j'(c_j), \quad i=1,\ldots,n.
$$
We show in Lemma \ref{lem.AK} that this system can be inverted, leading to
$$
  \sqrt{c_i}u_i = -\sum_{j=1}^n B_{ij}^{BD}(\bm{c})\sqrt{c_j}\na h'_j(c_j),
$$
where $B^{BD}(\bm{c})$ is the Bott--Duffin inverse of $B(\bm{c})$,
and system \eqref{1.gms1}--\eqref{1.gms2} can be formulated as
$$
  \pa_t c_i = \diver\bigg(\sum_{j=1}^n\sqrt{c_i}B_{ij}^{BD}(\bm{c})\sqrt{c_j}
	\na\frac{\delta H}{\delta c_j}(\bm{c})\bigg), \quad i=1,\ldots,n,
$$
which, by the way, equals \eqref{1.concise} if $H(\bm{c})$ is given by
\eqref{1.ent} and $B(\bm{c})=A(\bm{c})$.

We suppose for all $\bm{c}\in[0,1]^n$ the following conditions on the matrix
$B(\bm{c})$:

\renewcommand{\labelenumi}{(B\theenumi)}
\begin{enumerate}
\item $B(\bm{c})$ is symmetric and $L=\ran B(\bm{c})$,
$L^\perp=\ker(B(\bm{c})P_L)$.
\item For all $i,j=1,\ldots,n$ and $s>0$, $B_{ij}(\bm{c})$ is bounded and
Lipschitz continuous for all $\bm{c}\in[s,1]^n$.
\item There exists a function $\gamma:(0,\infty)\to(0,\infty)$ such that for all
$m>0$ and all $s\ge m$, it holds that $\gamma(s)\le \gamma(m)$ and
$\|B(\bm{c})\|_F\le \gamma(\min_{i=1,\ldots,n}c_i)$.
\item All nonzero eigenvalues of $B(\bm{c})$ are not smaller than a positive
constant $\mu>0$.
\end{enumerate}
 
The partial free energy functions $h_i(c_i)$ are associated to the pressures 
$p_i (c_i)$ via the thermodynamic relations
\begin{equation}\label{4.defpress}
  p_i' (c_i) = c_i h_i'' (c_i), \quad
  p_i (c_i)  = c_i h_i'(c_i)  - h_i (c_i).
\end{equation}
For $h_i(c_i)$ and $p_i (c_i)$, we assume that, for some constants $K_1$, $K_2>0$,
it holds that
\begin{equation}\label{hyp.fnh}\tag{H}
  h_i \in C^3((0,1]), \quad 0 < c_i h_i''(c_i) \le K_1,\
	|p_i''(c_i)| \le K_2 h_i''(c_i)\mbox{ for }c_i \in (0,1]
\end{equation}
for $i=1,\ldots,n$. 
This hypothesis implies that $h_i (c_i)$ is strictly convex, $p_i (c_i)$ 
is Lipschitz on $(0,1]$. The functions $h_i (c_i) = c_i\log c_i - c_i$ and 
$h_i (c_i) = c_i^\gamma$, $\gamma > 1$, satisfy \eqref{hyp.fnh}.

Our final main result is the weak-strong uniqueness property for 
\eqref{1.gms1}--\eqref{1.gms2}.

\begin{theorem}[Weak-strong uniqueness for the generalized system]\label{thm.gws}
Let Assumptions (A1)--(A3) and (B1)--(B4) hold, and let $h_i$ satisfy 
Hypothesis (H).
Let $\bm{c}$ be a weak solution and $\bar{\bm{c}}$ be a strong solution
to \eqref{1.bic}, \eqref{1.gms1}--\eqref{1.gms2}. We suppose that $\bar{\bm{c}}$
satisfies $\bar{c}_i(t)\ge m$ in $\Omega$, $t>0$ for some constant $m>0$, 
$$
  h_i'(\bar{c}_i)\in H_{\rm loc}^1(\Omega\times(0,\infty))
	\cap L_{\rm loc}^\infty(0,\infty;W^{2,\infty}(\Omega)), \quad i=1,\ldots,n,
$$
and the entropy identity
\begin{equation}\label{1.entid}
  H(\bar{\bm{c}}(t)) + \sum_{i,j=1}^n\int_0^t\int_\Omega 
	\sqrt{\bar{c}_i\bar{c}_j}B_{ij}^{BD}(\bar{\bm{c}})
	\na h'_i(\bar{c}_i)\cdot\na h'_j(\bar{c}_j)dxds
	= H(\bar{\bm{c}}^0)
\end{equation}
for $t>0$. Then there exists a constant $C(t)>0$, depending on $t$, $m$, and
$(D_{ij})$, such that
$$
  H(\bm{c}(t)|\bar{\bm{c}}(t)) \le C(t)H(\bm{c}^0|\bar{\bm{c}}^0)
	\quad\mbox{for }t>0.
$$
If the initial data coincide, then
$\bm{c}(t)=\bar{\bm{c}}(t)$ in $\Omega$ for $t>0$.
\end{theorem}

We do not explore the existence of solutions with the stated regularity.
The existence of weak solutions to \eqref{1.bic}, \eqref{1.gms1}--\eqref{1.gms2}
can be shown by the techniques detailed in \cite{Jue15,Jue16} under suitable
assumptions on $K_{ij}$ and $h'_i$ that guarantee nonlinear
gradient estimates. The existence of (local-in-time) strong solutions can be
shown by following the approach of \cite{HMPW17} by formulating
\eqref{1.gms1}--\eqref{1.gms2} as
$$
  \pa_t c_i = {\mathcal M}_i(\bm{c}) 
	:= \diver\bigg(\sum_{j=1}^n M_{ij}(\bm{c})\na c_j\bigg), \quad
	i=1,\ldots,n,
$$
where $M_{ij}(\bm{c})$ depends on $B^{BD}(\bm{c})$ and $h''_i(c_i)$,
and verifying that the principal part of the operator ${\mathcal M}(\bm{c})
=({\mathcal M}_1,\ldots,{\mathcal M}_n)(\bm{c})$, 
defined on suitable spaces, is normally elliptic and satisfies the
Lopatinski--Shapiro condition. By \cite[Theorem 8.2]{DHP03}, the operator
${\mathcal M}(\bm{c})$ has maximal regularity of type $L^p$ and the local
existence result follows from \cite[Theorem A1]{HMPW17}.

The strategy of the proof of Theorem \ref{thm.gws}
is similar to that one of Theorem \ref{thm.ws}, but
it is more involved. First, we show a relative entropy inequality.
 The terms of this inequality are estimated by splitting the domain into two regions:
$c_* (x,t) :=\min_{i=1,\ldots,n}c_i(x,t)\le m/2$ and $c_* (x,t) >m/2$,
where $m>0$ is the uniform lower bound for $\bar{c}_i$. 
The final estimate reads
\begin{align}\label{1.rei3}
  \frac{dH}{dt}(\bm{c}|\bar{\bm{c}}) 
	&+ \frac12\sum_{i,j=1}^n\int_\Omega (1-\chi(\bm{c}))
	\sqrt{c_ic_j}B_{ij}^{BD}(\bm{c})\na h'_i(c_i)\cdot\na h'_j(c_j) dx \\
	&{}+ C(m)\sum_{i=1}^n\int_\Omega \chi(\bm{c})
	|\na(c_i-\bar{c}_i)|^2dx
	\le CH(\bm{c}|\bar{\bm{c}}), \nonumber
\end{align}
where $\chi(\bm{c})$ a cutoff function that vanishes if $c_i\le m/2$ for
some $i$ (see \eqref{4.defcutoff} for details).
An application of Gronwall's lemma completes the proof.
Notice, however,
that we do not obtain a gradient estimate as in \eqref{1.rei2}.

By specifying the coefficients $K_{ij}(\bm{c})$ and the entropy densities
$h_i$, we prove the weak-strong property for cross-diffusion systems
describing physical vapor deposition processes \cite{BaEh18}
and for the tumor-growth model suggested in \cite{JaBy02} 
and analyzed in \cite{JuSt13} {\color{PineGreen} and the Maxwell-Stefan system considering different molar masses that is derived in \cite{BCH54,BoDr15}}; see Section \ref{sec.exam}.

The main contributions of this work are, first, the derivation of the
relative entropy inequality \eqref{1.aux} for the Maxwell-Stefan system and 
\eqref{1.rei3} for generalized Maxwell-Stefan systems.
Second, the introduction of the
Bott--Duffin inverse provides an efficient way to reduce the Maxwell--Stefan
system to a (degenerate) parabolic system formulated in the square roots
$\sqrt{c_i}$. (Related formulations using the chemical potentials 
$\delta H/\delta c_i$ can be found in
\cite{BoDr20}.) Third, we show that our technique can be extended to more general
Maxwell--Stefan systems which may have degeneracy at zero. 

The paper is organized as follows. We study the properties of the matrix
$A(\bm{c})$, defined in \eqref{1.A}, and its Bott--Duffin inverse
$A^{BD}(\bm{c})$ in Section \ref{sec.A}. In Section \ref{sec.ws}, we recall
the existence results for global weak and local strong solutions to 
\eqref{1.eq}--\eqref{1.bic}, prove the relative entropy inequalities \eqref{1.rei}
and \eqref{1.aux} as well as Theorem \ref{thm.ws} and Corollary \ref{coro}.
Section \ref{sec.gen} is devoted to the existence of the Bott--Duffin inverse
of $B(\bm{c})$, defined in \eqref{1.AK}, and the proof of the 
relative entropy inequality
\eqref{1.rei3} eventually leading to the weak-strong uniqueness Theorem \ref{thm.gws}. 
In Section 5, we present some examples that fit into our framework.
Finally, we recall the definition and some properties of the Bott--Duffin inverse
in Appendix \ref{app}, show two simple inequalities for the Boltzmann entropy
density in Appendix \ref{app.ent}, and derive the generalized model
\eqref{1.gms1}--\eqref{1.gms2} from thermodynamic principles 
in Appendix \ref{app.thermo}. 

\subsection*{Notation}

We set $\R_+= [0,\infty)$. Elements of the matrix $A\in\R^{n\times n}$ are denoted 
by $A_{ij}$, $i,j=1,\ldots,n$, and the elements of a vector $\bm{c}\in\R^n$ are 
$c_1,\ldots,c_n$. If $f:\R\to\R$ is any function, we define
$f(\bm{c})=(f(c_1),\ldots,f(c_n))$ for $\bm{c}\in\R^n$. 
In the whole paper, $C>0$, $C_i>0$ denote generic constants
whose values change from line to line.


\section{Properties of the matrix $A(\bm{c})$}\label{sec.A}

The properties of the matrix $A(\bm{c})$, defined in \eqref{1.A}, have been
studied in \cite{Bot11,HMPW17,JuSt13} under the assumption $c_i>0$ for all 
$i=1,\ldots,n$. Our results are valid for nonnegative concentrations $c_i\ge 0$,
including vacuum.

Let $\bm{c}\in\R^n_+$. Since $(D_{ij})$ is symmetric, we have for all $\bm{z}\in\R^n$,
$$
  0 = \sum_{i,j=1}^n A_{ij}(\bm{c})z_j = \sum_{i,j=1,\,j\neq i}^n
	\frac{c_j}{D_{ij}}z_i - \sum_{i,j=1,\,j\neq i}^n 
	\frac{\sqrt{c_ic_j}}{D_{ij}}z_j = \sum_{i,j=1,\,j\neq i}^n
	\frac{\sqrt{c_j}}{D_{ij}}(\sqrt{c_j}z_i - \sqrt{c_i}z_j),
$$
showing that $\operatorname{span}\{\sqrt{\bm{c}}\}=\ker A(\bm{c})$. We set
\begin{align*}
  \ran A(\bm{c}) &= L := \{\bm{x}\in\R^n:\sqrt{\bm{c}}\cdot\bm{x} = 0\}, \\
	\ker A(\bm{c})  &= (\ran A(\bm{c}))^\perp = L^\perp 
	= \operatorname{span}\{\sqrt{\bm{c}}\},
\end{align*}
and note that $\sum_{i=1}^n c_i =1$ implies that 
$|\sqrt{\bm{c}}|^2 = c_1+\cdots+c_n = 1$. 
The projection matrices $P_L$ on $L$ and $P_{L^\perp}$ on $L^\perp$ are given by
\begin{equation}\label{2.PL}
  (P_L)_{ij} = \delta_{ij} - \sqrt{c_ic_j}, \quad 
	(P_{L^\perp})_{ij} = \delta_{ij} - (P_L)_{ij} = \sqrt{c_ic_j}, \quad
	i,j=1,\ldots,n.
\end{equation}

\begin{lemma}\label{lem.A}
Let $\bm{c}\in\R_+^n$ be such that $\sum_{i=1}^n c_i=1$. Then
\begin{equation}\label{2.A}
  \bm{z}^TA(\bm{c})\bm{z} \ge \mu|P_L\bm{z}|^2\quad\mbox{for all }\bm{z}\in\R^n,
\end{equation}
where $\mu=\min_{i\neq j}(1/D_{ij})$. Moreover, the Bott--Duffin inverse
$$
A^{BD}(\bm{c}) = P_L(A(\bm{c})P_L+P_{L^\perp})^{-1}
$$ 
is well defined, symmetric, and satisfies
\begin{equation}\label{2.ABD}
  \bm{z}^TA^{BD}(\bm{c})\bm{z} \ge \lambda|P_L\bm{z}|^2\quad\mbox{for all }
	\bm{z}\in\R^n,
\end{equation}
where $\lambda=(2\sum_{i\neq j}(1/D_{ij} + 1))^{-1}$. 
\end{lemma}

\begin{proof} 
We first prove \eqref{2.A}.
Let $0<\alpha\le\mu$ and suppose that $c_i\neq 0$
for $i=1,\ldots,M$ and $c_i=0$ for $i=M+1,\ldots,n$. 
If necessary, we may rearrange the indices to achieve this ordering. 
Since $\sum_{i=1}^n c_i=1$, it holds that $M>0$. 
Thus, we can write $-A(\bm{c})-\alpha P_{L^\perp}$ in block diagonal form as
$$
  -A(\bm{c})-\alpha P_{L^\perp} = \left(\begin{array}{c|ccc}
	\widetilde{A} & 0 & 0 & 0 \\ \hline
	0 & a_{M+1} & & 0 \\
	\vdots & 0  & \ddots & 0 \\
	0 & 0 & & a_n
	\end{array}\right),
$$
where $\widetilde{A}\in\R^{M\times M}$ has the coefficients
$\widetilde{A}_{ii}=-A_{ii}(\bm{c})-\alpha c_i$ and 
$\widetilde{A}_{ij}=(1/D_{ij}-\alpha)\sqrt{c_ic_j}$ for 
$i,j=1,\ldots,M$, $i\neq j$,
and $a_j=-\sum_{k=1,\,k\neq j}^n c_k/D_{kj}$ for $j=M+1,\ldots,n$. 
Because of $\alpha\le\mu$, the matrix $\widetilde{A}$ is quasi-positive and
irreducible. Hence, by the Perron--Frobenius theorem \cite[Chapter~8]{Mey00}, 
the spectral radius of $\widetilde{A}$ is less than or equal to
the Perron--Frobenius eigenvalue that is a simple eigenvalue of 
$\widetilde{A}$ associated with a strictly positive eigenvector,
and all other eigenvalues of $\widetilde{A}$ have no positive eigenvector. 
In the present case, the Perron--Frobenius eigenvalue is given by $\lambda_{\rm PF}
= -\alpha$ and is associated with the eigenvector $(\sqrt{c_1},\ldots,\sqrt{c_M})$, 
recalling that $c_i>0$ for all $i=1,\ldots,M$ (also see the proof of Lemma 2.1
in \cite{JuSt13}). Since all eigenvalues of $\widetilde{A}$ are not larger than
$\lambda_{\rm PF}$, we have
$$
  \widetilde{\bm{z}}^T(-\widetilde{A})\widetilde{\bm{z}} 
	\ge \alpha|\widetilde{\bm{z}}|^2\quad\mbox{for }
	\widetilde{\bm{z}}=(z_1,\ldots,z_M)\in\R^M.
$$
This leads, for any $\alpha\le \mu$ and $\bm{z}\in\R^n$, to the inequality
\begin{align*}
  \bm{z}^T&(A(\bm{c}) + \alpha P_{L^\perp})\bm{z}
	= \widetilde{z}^T(-\widetilde{A})\widetilde{\bm{z}}
	+ \sum_{i=M+1}^n\sum_{j=1,\,j\neq i}^n \frac{c_j}{D_{ij}}z_i^2 \\
	&\ge \alpha|\widetilde{\bm{z}}|^2 + \min_{\substack{k,\ell=1,\ldots,M \\ k\neq\ell}}
	\frac{1}{D_{k\ell}}\sum_{j=1,\,j\neq i}^n c_j\sum_{i=M+1}^n z_i^2 
	\ge \alpha|\widetilde{\bm{z}}|^2 + \alpha\sum_{i=M+1}^n z_i^2 = \alpha|\bm{z}|^2,
\end{align*}
where we have used the fact that $\sum_{j=1,\,j\neq i}^n c_j = 1$ for
$i=M+1,\ldots,n$, since $c_i=0$ for exactly these indices.
This inequality implies that for all $\bm{z}=P_L\bm{z}+P_{L^\perp}\bm{z}\in\R^n$,
$$
  \bm{z}^TA(\bm{c})\bm{z} + \alpha|P_{L^\perp}\bm{z}|^2
	= \bm{z}^T(A(\bm{c}) + \alpha P_{L^\perp})\bm{z}
	\ge \alpha|P_L\bm{z}|^2 + \alpha|P_{L^\perp}\bm{z}|^2,
$$
which shows \eqref{2.A}.

The invertibility of $A(\bm{c})P_L+P_{L^\perp}$ is a consequence of Lemma 
\ref{lem.bott2} in the appendix. Consequently, the Bott--Duffin inverse
$A^{BD}(\bm{c})=P_L(A(\bm{c})P_L+P_{L^\perp})^{-1}$ exists. 

It remains to show \eqref{2.ABD}. 
The spectral radius $r(A(\bm{c})P_L+P_{L^\perp})$
is bounded by the Frobenius norm. Thus, because of $A(\bm{c})P_L = A(\bm{c})$
(see Lemma \ref{lem.bott2} in Appendix \ref{app}) and $0\le c_i\le 1$,
\begin{align*}
  r(A(\bm{c})P_L+P_{L^\perp}) &\le \|A(\bm{c})+P_{L^\perp}\|_F
	= \bigg(\sum_{i,j=1}^n(A_{ij}(\bm{c}) + \sqrt{c_ic_j})^2\bigg)^{1/2} \\
	&= \bigg\{\sum_{i=1}^n\bigg(\sum_{j=1,\,j\neq i}^n \frac{c_j}{D_{ij}}+c_i\bigg)^2
	+ \sum_{i,j=1,\,i\neq j}^n\bigg(1-\frac{1}{D_{ij}}\bigg)^2 c_ic_j\bigg\}^{1/2} \\
	&\le 2\sum_{i,j=1,\,i\neq j}\bigg(\frac{1}{D_{ij}}+1\bigg) = \frac{1}{\lambda}.
\end{align*}
We infer that the eigenvalues of $(A(\bm{c})P_L+P_{L^\perp})^{-1}$ are larger than
or equal to $\lambda$. Thus, in view of \eqref{app.bd}, we find that for all
$\bm{z}\in\R^n$,
$$
  \bm{z}^T A^{BD}(\bm{c})\bm{z} = (P_L\bm{z})^T(A(\bm{c})P_L+P_{L^\perp})^{-1}
	P_L\bm{z} \ge \lambda|P_L\bm{z}|^2,
$$
finishing the proof.
\end{proof}

Since $(\na\sqrt{c_1},\ldots,\na\sqrt{c_n})\in L$, the existence of the 
Bott--Duffin inverse guarantees that the solution of \eqref{1.eqA} 
can be expressed via the formula \eqref{1.eqABD}; see Appendix \ref{app}. 


\section{Weak-strong uniqueness for Maxwell--Stefan systems}\label{sec.ws}

\subsection{Existence theory}\label{sec.ex.ms}

We discuss the existence of weak and strong solutions to the Maxwell--Stefan system
\eqref{1.eq}--\eqref{1.bic}. First, we recall
the existence theorem for weak solutions, which was proved in \cite{JuSt13}.

\begin{theorem}[Global existence for Maxwell--Stefan systems]\label{thm.weak}
Let Assumptions (A1)--(A3) hold. Then there exists a weak solution to
\eqref{1.eq}--\eqref{1.bic} satisfying the entropy inequality \eqref{1.entropy} 
for $t>0$, or equivalently,
\begin{align*}
	H(\bm{c}(t)) + 4\sum_{i,j=1}^n\int_0^t\int_\Omega A^{BD}_{ij}(\bm{c})
	\na\sqrt{c_i}\cdot \na\sqrt{c_j} dxds
	&\le H(\bm{c}_0).
\end{align*}
\end{theorem}

The existence of strong solutions was proved in \cite[Theorem 1]{Bot11}
and \cite[Theorem 3.2]{HMPW17}.

\begin{theorem}[Strong solutions for Maxwell--Stefan systems]\label{thm.strong}
Let $\Omega\subset\R^d$ $(d\ge 1)$ be a bounded domain with $\pa\Omega\in C^2$
and let $\bm{c}^0\in W^{2-2/p,p}(\Omega;\R^n)$ with $c_{i}^0\ge 0$,
$\sum_{i=1}^n c_i^0=1$ in $\Omega$, where $p>d+2$. 
Then there exists $T^*>0$ and a unique
solution $\bm{c}$ to \eqref{1.eq}--\eqref{1.bic} satisfying
$$
  c_i\in C^1((0,T^*);W^{2-2/p,p}(\Omega))\cap W^{1,p}(0,T;L^p(\Omega))\cap
	L^p(0,T;W^{2,p}(\Omega))
$$
for $i=1,\ldots,n$.
\end{theorem}

The strong solution of Theorem \ref{thm.strong} has the property of immediate
positivity: If $c_{i}^0\ge 0$ in $\Omega$ then $c_i(t)>0$ in $\Omega$
for $0<t<T'$, where $T'\le T$ depends on $\bm{c}^0$. Moreover, if the initial
data is close to a constant vector, the strong solution can be extended globally:
Let $\bm{c}^*\in\R_+^n$. Then there exists $\eps>0$ such that if the initial
data satisfies $\|\bm{c}^0-\bm{c}^*\|_{W^{2-2/p,p}(\Omega)}\le\eps$, then
the strong solution exists globally in time.

If $c_{i}^0>0$ in $\Omega$ for $i=1,\ldots,n$, the continuity of the strong
solution implies
that there exists $0<T''\le T^*$ and $m>0$ such that $c_i(t)\ge m>0$
in $\Omega$ for $i=1,\ldots,n$. Therefore, because of the embedding
$W^{2-2/p,p}(\Omega)\hookrightarrow C^1(\overline{\Omega})$, we have
$\sqrt{c_i}\in L^\infty(0,T'';W^{1,\infty}(\Omega))$, and $T''=\infty$
if $\|\bm{c}^0-\bm{c}^*\|_{W^{2-2/p,p}(\Omega)}$ is sufficiently small.
This shows that the strong solution satisfies the regularity assumptions
of Corollary \ref{coro}.

The assumption $\sqrt{c_i}\in L_{\rm loc}^\infty(0,\infty;W^{1,\infty}(\Omega))$
and the property $c_i(t)\ge m$ in $\Omega$ imply that the regularity condition 
$u_i\in L_{\rm loc}^\infty(\Omega\times(0,\infty))$ of Theorem \ref{thm.ws} 
is satisfied.
This is a consequence of the following lemma and
$m\sum_{i=1}^n |u_i|^2\le \sum_{i=1}^n c_i|u_i|^2\le C\sum_{i=1}^n|\na\sqrt{c_i}|^2$.
Moreover, the assumption $\log c_i\in H_{\rm loc}^1(\Omega\times(0,\infty))$
follows from $|\na\log c_i|\le |\na c_i|/m\in C^0((0,T'');C^0(\overline\Omega))$.

\begin{lemma}\label{lem.cu2}
Let $0\le c_i\le 1$ and let $u_i$ be given by the force-flux relations in
\eqref{1.eq} satisfying $\sum_{i=1}^n c_iu_i=0$. Then
there exists a constant $C>0$, only depending on $(D_{ij})$ such that
$$
  \sum_{i=1}^n c_i|u_i|^2 \le C\sum_{i=1}^n|\na\sqrt{c_i}|^2.
$$
\end{lemma}

\begin{proof}
It follows from \eqref{1.eqA} and the symmetry of $A(\bm{c})$, defined in \eqref{1.A},
that
$$
  4\sum_{i=1}^n|\na\sqrt{c_i}|^2 = \sum_{i=1}^n\bigg|\sum_{j=1}^n A_{ij}(\bm{c})
	\sqrt{c_j}u_j\bigg|^2 = \sum_{i,j=1}^n \sqrt{c_i}u_i(A(\bm{c})^2)_{ij}
	\sqrt{c_j}u_j.
$$
Since the eigenvalues of $A(\bm{c})^2$ are the square of the eigenvalues of
$A(\bm{c})$, we deduce from \eqref{2.A} that
$\bm{z}^T A(\bm{c})^2\bm{z} \ge \mu^2|P_L\bm{z}|^2$ for all $\bm{z}\in\R^n$.
This yields
$$
  4\sum_{i=1}^n|\na\sqrt{c_i}|^2 \ge \mu^2|P_L(\sqrt{c_i}u_i)_{i=1,\ldots,n}|^2.
$$
Because of $\sum_{i=1}^n\sqrt{c_i}(\sqrt{c_i}u_i)=0$, we have $(\sqrt{c_i}u_i)_i\in L$
and hence, $P_L(\sqrt{c_i}u_i)_i=(\sqrt{c_i}u_i)_i$. The statement of the lemma
follows after setting $C=4/\mu^2$.
\end{proof}


\subsection{Relative entropy inequality}\label{sec.relent}

We first derive a relative entropy inequality via a formal computation.
Using \eqref{1.relent} and \eqref{1.eq}, we obtain
\begin{align}\label{2.aux0}
  \frac{d}{dt}H(\bm{c}|\bar{\bm{c}})
	&= \sum_{i=1}^n\int_\Omega\bigg(\log\frac{c_i}{\bar{c}_i}\pa_t c_i
	+ \bigg(1-\frac{c_i}{\bar{c}_i}\bigg)\pa_t\bar{c}_i\bigg)dx \\
	&= \sum_{i=1}^n\int_\Omega\bigg(\na\log\frac{c_i}{\bar{c}_i}\cdot (c_iu_i)
	- \na\log\frac{c_i}{\bar{c}_i}\cdot(c_i\bar{u}_i)\bigg)dx \nonumber \\
	&= \sum_{i=1}^n\int_\Omega c_i\na(\log c_i-\log\bar{c}_i)\cdot(u_i-\bar{u}_i)dx.
	\nonumber
\end{align}
To reformulate the integrand of the right-hand side, we insert the second equation
of \eqref{1.eq}, and use the symmetry of $(D_{ij})$:
\begin{align}
   \sum_{i=1}^n &c_i\na(\log c_i-\log\bar{c}_i)\cdot(u_i-\bar{u}_i)
	= -\sum_{i=1}^n c_i(u_i-\bar{u}_i)\cdot\sum_{j\neq i}\frac{1}{D_{ij}}
	\big(c_j(u_i-u_j) - \bar{c}_j(\bar{u}_i-\bar{u}_j)\big) \nonumber \\
	&= -\sum_{i=1}^nc_i(u_i-\bar{u}_i)\cdot\sum_{j\neq i}\frac{c_j}{D_{ij}}
	\big((u_i-\bar{u}_i)-(u_j-\bar{u}_j)\big)  \nonumber \\
	&\phantom{xx}{}
	- \sum_{i,j=1,\,i\neq j}^n\frac{1}{D_{ij}}c_i(u_i-\bar{u}_i)\cdot\big((c_j-\bar{c}_j)
	(\bar{u}_i-\bar{u}_j)\big) \label{3.ed} \\
	&= -\sum_{i,j=1,\,i\neq j}^n\frac{c_ic_j}{2D_{ij}}\big|(u_i-\bar{u}_i)
	-(u_j-\bar{u}_j)\big|^2
	- \sum_{i,j=1,\,i\neq j}^n\frac{c_i}{D_{ij}}
	(c_j-\bar{c}_j)(u_i-\bar{u}_i)\cdot(\bar{u}_i-\bar{u}_j). \nonumber
\end{align}
This shows that
\begin{align}\label{2.aux1}
  \frac{d}{dt}H&(\bm{c}|\bar{\bm{c}})
	+ \frac12\sum_{i,j=1,\,i\neq j}^n\int_\Omega\frac{c_ic_j}{D_{ij}}\big|(u_i-\bar{u}_i)
	-(u_j-\bar{u}_j)\big|^2 dx \\
	&{}= -\sum_{i,j=1,\,i\neq j}^n\int_\Omega\frac{c_i}{D_{ij}}
	(c_j-\bar{c}_j)(u_i-\bar{u}_i)\cdot(\bar{u}_i-\bar{u}_j)dx. \nonumber
\end{align}

Our aim is to make this computation rigorous. Since the computation in 
\eqref{3.ed} is purely algebraic, it holds without any regularity restrictions.
In principle, one would expect that \eqref{2.aux1} holds under the condition
that all the terms are well defined, which would cover the class of weak solutions
subject to the condition $u_i\in L^\infty$ (to ensure integrability of the right-hand
side). However, we have not been able to establish \eqref{2.aux0} for such a
class of solutions, and stricter conditions on one of the solutions are
required.

\begin{lemma}\label{lem.relent}
Let $\bm{c}$ be a weak solution to \eqref{1.eq}--\eqref{1.bic} and let $\bar{\bm{c}}$
be a strong solution to \eqref{1.eq}--\eqref{1.bic} satisfying $0<\bar{c}_i(t)<1$
in $\Omega$, the regularity
$$
  \log\bar{c}_i\in L_{\rm loc}^2(0,\infty;H^1(\Omega)), \
	\pa_t\log\bar{c}_i\in L_{\rm loc}^2(\Omega\times(0,\infty)), \
	\bar{u}_i\in L_{\rm loc}^\infty(0,\infty;L^\infty(\Omega)),
$$
and the entropy identity
\begin{equation}\label{2.entid}
  H(\bar{\bm{c}}(t)) + \sum_{i,j=1}^n\frac12\int_0^t\int_\Omega
	\frac{\bar{c}_i\bar{c}_j}{D_{ij}}|u_i-\bar{u}_i|^2 dxds = H(\bar{\bm{c}}^0)
	\quad\mbox{for }t>0.
\end{equation}
Then
\begin{align}\nonumber
  H(\bm{c}(t)&|\bar{\bm{c}}(t)) + \sum_{i,j=1}^n\int_0^t\int_\Omega
	A_{ij}(\bm{c}) (\sqrt{c_i}(u_i-\bar{u}_i))\cdot(\sqrt{c_j}(u_j-\bar{u}_j)) dxds \\
	&{}\le H(\bm{c}^0|\bar{\bm{c}}^0)
	 -\sum_{i,j=1,\,i\neq j}^n   \int_0^t   \int_\Omega\frac{c_i}{D_{ij}}
	(c_j-\bar{c}_j)(u_i-\bar{u}_i)\cdot(\bar{u}_i-\bar{u}_j)dx ds.
\label{3.Hrel}
\end{align}
\end{lemma}

\begin{proof}
Since
$$
  H(\bm{c}|\bar{\bm{c}}) = H(\bm{c}) - H(\bar{\bm{c}}) 
	- \int_{\Omega} \sum_{i=1}^n(c_i-\bar{c}_i)\log\bar{c}_i dx,
$$
we need to formulate the time evolution of each of these terms.
According to Theorem \ref{thm.weak}, the weak solution $\bm{c}$ satisfies
$\na\sqrt{c_i}$, $\sqrt{c_i}u_i\in L_{\rm loc}^2(0,\infty;L^2(\Omega))$ and
$$
  H(\bm{c}(t)) + \frac12\sum_{i,j=1}^n\int_0^t\int_\Omega
	\frac{c_ic_j}{D_{ij}}|u_i-u_j|^2 dxds \le H(\bm{c}^0)\quad\mbox{for }t>0.
$$
The symmetry of $(D_{ij})$ and the force-flux relations in \eqref{1.eq}  give
$$
  \sum_{i=1}^n c_iu_i\cdot\na\log c_i
	= -\sum_{i,j=1}^n c_iu_i\cdot\frac{c_j}{D_{ij}}(u_i-u_j)
	= -\frac12\sum_{i,j=1}^n\frac{c_ic_j}{D_{ij}}|u_i-u_j|^2 , 
$$
and we formulate the  entropy inequality as
\begin{equation}\label{3.Hc}
  H(\bm{c}(t))  - H(\bm{c}^0)  \le   \sum_{i=1}^n\int_0^t\int_\Omega c_iu_i\cdot\na\log c_i dxds
	 \quad\mbox{for }t>0.
\end{equation}
The expression $c_iu_i\cdot\na\log c_i$ has to be understood
as $2\na\sqrt{c_i}\cdot(\sqrt{c_i}u_i)$, which is well defined since
$\na\sqrt{c_i}$, $\sqrt{c_i}u_i\in L^2(\Omega\times(0,T))$ (see Lemma \ref{lem.cu2}).
In a similar way, we express the entropy identity \eqref{2.entid} as
\begin{equation}\label{3.Hbarc}
  H(\bar{\bm{c}}(t))  -  H(\bar{\bm{c}}^0) = \sum_{i=1}^n\int_0^t\int_\Omega \bar{c}_i\bar{u}_i\cdot \na\log \bar{c}_i dxds
	\quad\mbox{for }t>0.
\end{equation}

Next, the difference of the weak formulations for $\bm{c}$ and $\bar{\bm{c}}$ gives
\begin{align*}
  \int_\Omega& (c_i-\bar{c}_i)(t)\phi_i(t) dx 
	- \int_\Omega(c_{i}^0-\bar{c}_{i}^0)\phi_i (0)dx \\
	&= \int_0^t\int_\Omega(c_i-\bar{c}_i)\pa_t\phi_i dxds 
	+\int_0^t\int_\Omega(c_iu_i-\bar{c}_i\bar{u}_i)\cdot\na\phi_i dxds
\end{align*}
for test functions $\phi_i\in C_{\rm loc}^1([0,\infty);C^1(\overline\Omega))$.
Using a density argument we see that the test function
$\phi_i$ can be taken in the class $H^1(\Omega\times(0,T))$ for $T > 0$, 
in which case $\phi_i (t)$, $\phi_i(0)$ are well defined by the trace theorem.
Selecting $\phi_i=\log\bar{c}_i$, we obtain
\begin{align}\label{3.aux1}
  \int_\Omega & (c_i-\bar{c}_i)(t)\log\bar{c}_i(t) dx 
	- \int_\Omega(c_{i}^0-\bar{c}_{i}^0)\log\bar{c}_{i}^0dx \\
	&= \int_0^t\int_\Omega(c_i-\bar{c}_i)\frac{\pa_t\bar{c}_i}{\bar{c}_i} dxds 
	+ \int_0^t\int_\Omega(c_iu_i-\bar{c}_i\bar{u}_i)\cdot\na\log\bar{c}_i dxds. 
	\nonumber
\end{align}
Taking into account the regularity properties of $\bar{c}_i$, we insert
$\pa_t\bar{c}_i=-\diver(\bar{c}_i\bar{u}_i)$ in the third term and integrate by parts:
\begin{equation*}
  \int_0^t\int_\Omega(c_i-\bar{c}_i)\frac{\pa_t\bar{c}_i}{\bar{c}_i} dxds
	= \int_0^t\int_\Omega\na\bigg(\frac{c_i}{\bar{c}_i}\bigg)
	\cdot(\bar{c}_i\bar{u}_i)dxds.
\end{equation*}
We wish to write the integrand on the right-hand side as 
$$
  \na\bigg(\frac{c_i}{\bar{c}_i}\bigg)\cdot(\bar{c}_i\bar{u}_i)
	= \bigg(\na c_i - \frac{c_i}{\bar{c}_i}\na\bar{c}_i\bigg)\cdot\bar{u}_i
	= c_i\na\log\bigg(\frac{c_i}{\bar{c}_i}\bigg)\cdot\bar{u}_i.
$$
Since $c_i\ge 0$ only, the expression $\log c_i$ may be not integrable.
Therefore, we define
$$
  \na\log\bigg(\frac{c_i}{\bar{c}_i}\bigg) 
	:= \frac{1}{\sqrt{c_i}}(2\na\sqrt{c_i} - \sqrt{c_i}\na\log\bar{c}_i)
	\quad\mbox{if }c_i>0
$$
as the product of two functions and $\na\log(c_i/\bar{c}_i)$ arbitrary if $c_i=0$.
Although this product may be not integrable, the expression
$c_i\na\log(c_i/\bar{c}_i)$ lies in $L^2(\Omega\times(0,T))$ and consequently,
$c_i\na\log(c_i/\bar{c}_i)\cdot\bar{u}_i$ lies in the same space. Therefore,
we can formulate \eqref{3.aux1} as
\begin{align}\label{3.aux2}
  \int_\Omega & (c_i-\bar{c}_i)(t)\log\bar{c}_i(t) dx 
	- \int_\Omega(c_{i}^0-\bar{c}_{i}^0)\log\bar{c}_{i}^0dx \\
	&= \int_0^t\int_\Omega c_i\na\log\bigg(\frac{c_i}{\bar{c}_i}\bigg)
	\cdot\bar{u}_i dxds
	+ \int_0^t\int_\Omega(c_iu_i-\bar{c}_i\bar{u}_i)\cdot\na\log\bar{c}_i dxds.
	\nonumber
\end{align}
Subtracting \eqref{3.Hbarc} and \eqref{3.aux2} from \eqref{3.Hc} leads to
\begin{equation}\label{3.aux21}
  H(\bm{c}(t)|\bar{\bm{c}}(t)) - H(\bm{c}^0|\bar{\bm{c}}^0)
	\le \sum_{i=1}^n\int_0^t\int_\Omega c_i\na\log\bigg(\frac{c_i}{\bar{c}_i}\bigg)
	\cdot(u_i-\bar{u}_i) dxds.
\end{equation}
Finally, using \eqref{3.ed} and the form \eqref{1.fricdiss} of the friction, 
we obtain \eqref{3.Hrel}. 
\end{proof}


\subsection{Proof of Theorem \ref{thm.ws}}

We proceed to estimate \eqref{3.Hrel}. We set
$\bm{Y}=(Y_1,\ldots,Y_n)$ with $Y_i=\sqrt{c_i}(u_i-\bar{u}_i)$, $i=1,\ldots,n$.
Then, using \eqref{2.A}, we have 
\begin{equation}\label{3.YAY}
  \sum_{i,j=1}^n A_{ij}(\bm{c})(\sqrt{c_i}(u_i-\bar{u}_i))\cdot(\sqrt{c_j}
	(u_j-\bar{u}_j)) = \bm{Y}^T A(\bm{c})\bm{Y} \ge \mu|P_L\bm{Y}|^2.
\end{equation}
It follows from the constraints $\sum_{i=1}^n c_iu_i
=\sum_{i=1}^n \bar{c}_i\bar{u}_i=0$ that
\begin{align}
  (P_{L^\perp} \bm{Y})_i &=  \sum_{j=1}^n \sqrt{c_i} c_j (u_j -\bar{u}_j ) 
	= \sqrt{c_i} \sum_{j=1}^n (\bar{c}_j - c_j)\bar{u}_j, \nonumber \\
  |P_L \bm{Y}|^2 &= | \bm{Y}|^2 - |P_{L^\perp} \bm{Y}|^2
  = \sum_{i=1}^n c_i|u_i-\bar{u}_i|^2 - \sum_{i=1}^n c_i 
	\bigg|\sum_{j=1}^n(c_j-\bar{c}_j)\bar{u}_j\bigg|^2 \nonumber \\ 
	&\ge \sum_{i=1}^n c_i|u_i-\bar{u}_i|^2
	- n \| \bar{u} \|_{L^\infty} \sum_{j=1}^n(c_j-\bar{c}_j)^2, \nonumber
\end{align}
where $\|\bar{u}\|_{L^\infty} := \max_{j=1,\ldots,n}\|u_j\|_{L^\infty
(\Omega\times(0,T))}$, and we used $\sum_{i=1}^n c_i=1$.

We turn to the last term in \eqref{3.Hrel}, which is estimated
\begin{align}
\label{3.aux12}
  \bigg|\int_0^t&\int_\Omega\sum_{i,j=1,\,i\neq j}^n  \frac{c_i}{D_{ij}}
	(c_j-\bar{c}_j)(u_i-\bar{u}_i)\cdot(\bar{u}_i-\bar{u}_j)dx ds\bigg| \\
  &\le \int_0^t\int_\Omega \sum_{i=1}^n \big(\sqrt{c_i} (u_i - \bar{u}_i\big) 
	\bigg(\sqrt{c_i}\sum_{j=1}^n\frac{|c_j - \bar{c}_j|}{D_{ij}}|\bar{u}_i - \bar{u}_j| 
	\bigg) dxds \nonumber \\
  &\le \frac{2\|\bar{u}\|_{L^\infty}}{\min_{i\neq j}D_{ij}}\int_0^t\int_\Omega  
	\bigg(\sum_{i=1}^n c_i|u_i - \bar{u}_i |^2\bigg)^{1/2} 
  \bigg(n\sum_{j=1}^n|c_j - \bar{c}_j|^2\bigg)^{1/2} dxds \nonumber \\
  &\le \frac{\mu}{2} \int_0^t\int_\Omega  \sum_{i=1}^n c_i|u_i-\bar{u}_i|^2 dxds
	+ C(\mu) \int_0^t\int_\Omega \sum_{i=1}^n (c_i-\bar{c}_i)^2dxds, \nonumber
\end{align}
where the constant $C(\mu)>0$ also depends on $\min_{i\neq j} D_{ij}$ and 
$\|\bar{u}\|_{L^\infty}$.
Inserting \eqref{3.YAY}--\eqref{3.aux12} into \eqref{3.Hrel} and taking
into account Lemma \ref{lem.ent} in Appendix \ref{app.ent}, we find that
\begin{align*}
  H(&\bm{c}(t)|\bar{\bm{c}}(t)) + \frac{\mu}{2}\sum_{i=1}^n\int_0^t\int_\Omega
	c_i|u_i-\bar{u}_i|^2 dxds \\
	&\le H(\bm{c}^0|\bar{\bm{c}}^0) + C(\mu)\int_0^t\int_\Omega \sum_{i=1}^n 
	(c_i-\bar{c}_i)^2dxds
	\le H(\bm{c}^0|\bar{\bm{c}}^0) + 2C(\mu) \int_0^t H(\bm{c}|\bar{\bm{c}})ds,
\end{align*}
and an application of Gronwall's lemma finishes the proof. 

\subsection{Proof of Corollary \ref{coro}}

In this section, we express the relative entropy via the Bott--Duffin inverse 
$A^{BD}(\bm{c})$. 

\begin{lemma}\label{lem.relent2}
Let the assumptions of Lemma \ref{lem.relent} hold with $\bar{u}_i\in 
L^\infty_{\rm loc}(0,\infty;L^\infty(\Omega))$ replaced by
$\sqrt{\bar{c}_i}\in L_{\rm loc}^\infty(0,\infty;W^{1,\infty}(\Omega))$. Then,
setting $\bm{Z}=(Z_1,\ldots,Z_n)$ with 
$Z_i=\na\sqrt{c_i}-(\sqrt{c_i}/\sqrt{\bar{c}_i})\na\sqrt{\bar{c}_i}$ 
for $i=1,\ldots,n$,
\begin{align}\label{3.Hrel2}
  H(&\bm{c}(t)|\bar{\bm{c}}(t)) + 4\sum_{i,j=1}^n\int_0^t\int_\Omega
	A_{ij}^{BD}(\bm{c})Z_i\cdot Z_j dxds \\
	&\le H(\bm{c}^0|\bar{\bm{c}}^0) + 4\sum_{i,j=1}^n\int_0^t\int_\Omega
	Z_i\cdot\na\sqrt{\bar{c}_j}\bigg(
	\frac{\sqrt{c_i}}{\sqrt{\bar{c}_i}}A_{ij}^{BD}(\bar{\bm{c}})
	- A_{ij}^{BD}(\bm{c})\frac{\sqrt{c_j}}{\sqrt{\bar{c}_j}}\bigg)dxds. 
	\nonumber
\end{align} 
\end{lemma}

\begin{proof}
Starting with the relative entropy inequality in the form \eqref{3.aux21}, 
we express its right-hand side by using \eqref{1.eqABD}: 
\begin{align*}
  \sum_{i=1}^n  &c_i\na(\log c_i-\log\bar{c}_i)\cdot(u_i-\bar{u}_i) \\
  &=
  - \sum_{i=1}^n \bigg( \nabla c_i - \frac{c_i}{\bar{c}_i} \nabla \bar{c}_i \bigg) 
	\cdot \sum_{j=1}^n \bigg( \frac{1}{\sqrt{c_i}}  A_{ij}^{BD} (\bm{c}) 
	\frac{1}{\sqrt{c_j}} \nabla c_j - \frac{1}{\sqrt{\bar{c}_i}}  
	A_{ij}^{BD} (\bar{\bm{c}})  \frac{1}{\sqrt{\bar{c}_j}} \nabla \bar{c}_j\bigg) \\
  &=
  -4 \sum_{i,j=1}^n Z_i \cdot A_{ij}^{BD} (\bm{c})  Z_j - 4 \sum_{i,j=1}^n 
	Z_i \cdot \bigg(A_{ij}^{BD} (\bm{c})\sqrt{ \frac{c_j}{\bar{c}_j}} 
	-  A_{ij}^{BD} (\bar{\bm{c}}) \sqrt{\frac{c_i}{\bar{c}_i}} \bigg)  
	\nabla \sqrt{\bar{c}_j},
\end{align*}
which gives \eqref{3.Hrel2}. 
\end{proof}

We continue with the proof of Corollary \ref{coro}. We estimate the two integrals
of the relative entropy inequality \eqref{3.Hrel2}. The integrand of the
second term is estimated, because of \eqref{2.ABD}, as
\begin{equation}\label{3.aux3}
  \sum_{i,j=1}^n A_{ij}^{BD}(\bar{c})Z_i\cdot Z_j \ge \lambda|P_L\bm{Z}|^2.
\end{equation}
The definitions of $P_L$ and $\bm{Z}$ yield
\begin{align*}
  (P_L\bm{Z})_i &= \bigg(\na\sqrt{c_i}-\frac{\sqrt{c_i}}{\sqrt{\bar{c}_i}}
	\na\sqrt{\bar{c}_i}\bigg)
	- \sum_{j=1}^n\sqrt{c_ic_j}\bigg(\na\sqrt{c_j}-\frac{\sqrt{c_j}}{\sqrt{\bar{c}_j}}
	\na\sqrt{\bar{c}_j} \bigg) \\
	&= \na(\sqrt{c_i}-\sqrt{\bar{c}_i}) 
	+ \frac{\sqrt{\bar{c}_i}-\sqrt{c_i}}{\sqrt{\bar{c}_i}}\na\sqrt{\bar{c}_i}
	- \sqrt{c_i}\sum_{j=1}^n\frac{(\sqrt{\bar{c}_j})^2-(\sqrt{c_j})^2}{\sqrt{\bar{c}_j}}
	\na\sqrt{\bar{c}_j}. 
\end{align*}
Using Young's inequality $(A+B+C)^2 \ge A^2/2 - 4B^2 - 4C^2$ and the
bounds $\bar{c}_i\ge m$ and $\sqrt{c_i}+\sqrt{\bar{c}_i}\le 2$, we infer that
\begin{align*}
  |P_L\bm{Z}|^2 &\ge \sum_{i=1}^n\bigg(\frac12|\na(\sqrt{c_i}-\sqrt{\bar{c}_i})|^2
	- 4\bigg|\frac{\sqrt{\bar{c}_i}-\sqrt{c_i}}{\sqrt{\bar{c}_i}}\na\sqrt{\bar{c}_i}
	\bigg|^2 \\
	&\phantom{xx}{}- 4\bigg|\sqrt{c_i}\sum_{j=1}^n\frac{(\sqrt{\bar{c}_j})^2
	-(\sqrt{c_j})^2}{\sqrt{\bar{c}_j}}\na\sqrt{\bar{c}_j}\bigg|^2\bigg) \\
	&\ge \frac12\sum_{i=1}^n|\na(\sqrt{c_i}-\sqrt{\bar{c}_i})|^2
	- \frac{4(n+1)}{m}\sum_{i=1}^n(\sqrt{c_i}-\sqrt{\bar{c}_i})^2|\na\sqrt{\bar{c}_i}|^2.
\end{align*}
With this estimate, \eqref{3.aux3} becomes, after integration over $\Omega\times(0,t)$,
\begin{align}\label{3.T1}
  4&\sum_{i,j=1}^n\int_0^t\int_\Omega A_{ij}^{BD}(\bar{c})Z_i\cdot Z_j dxds
	\ge 2\lambda\sum_{i=1}^n\int_0^t\int_\Omega|\na(\sqrt{c_i}-\sqrt{\bar{c}_i})|^2 
	dxds \\
	&\phantom{xx}{}- \frac{16(n+1)}{m}\max_{j=1,\ldots,n}\|\na\sqrt{\bar{c}_j}\|_{
	L^\infty(\Omega\times(0,T))}^2\sum_{i=1}^n\int_0^t\int_\Omega
	(\sqrt{c_i}-\sqrt{\bar{c}_i})^2 dxds. \nonumber 
\end{align}

Next, we consider the last integral in \eqref{3.Hrel2}.
By Young's inequality,
\begin{align}\label{3.aux4}
  4&\sum_{i,j=1}^n\int_0^t\int_\Omega
	Z_i\cdot\na\sqrt{\bar{c}_j}\bigg(\frac{\sqrt{c_i}}{\sqrt{\bar{c}_i}}
	A_{ij}^{BD}(\bar{\bm{c}})
	- A_{ij}^{BD}(\bm{c})\frac{\sqrt{c_j}}{\sqrt{\bar{c}_j}}\bigg)dxds
	\le \frac{\lambda}{2}\sum_{i=1}^n\int_0^t\int_\Omega|Z_i|^2 dxds \\
	&{}+ C\max_{j=1,\ldots,n}\|\na\sqrt{\bar{c}_j}\|_{
	L^\infty(\Omega\times(0,T))}^2\sum_{i,j=1}^n\int_0^t\int_\Omega
	\bigg(\frac{\sqrt{c_i}}{\sqrt{\bar{c}_i}}A_{ij}^{BD}(\bar{\bm{c}})
	- A_{ij}^{BD}({\bm{c}})\frac{\sqrt{c_j}}{\sqrt{\bar{c}_j}}\bigg)^2 dxds,
	\nonumber
\end{align}
and the constant $C>0$ depends on $\lambda$ and $n$. The first term on the 
right-hand side is estimated according to
\begin{align}\label{3.Z1}
  |Z_i|^2 &= \bigg|\na(\sqrt{c_i}-\sqrt{\bar{c}_i}) 
	- \frac{\sqrt{c_i}-\sqrt{\bar{c}_i}}{\sqrt{\bar{c}_i}}\na\sqrt{\bar{c}_i}\bigg|^2 \\
	&\le 2|\na(\sqrt{c_i}-\sqrt{\bar{c}_i})|^2
	+ \frac{2}{m}|\sqrt{c_i}-\sqrt{\bar{c}_i}|^2|\na\sqrt{\bar{c}_i}|^2. \nonumber
\end{align}

To estimate the second term on the right-hand side of \eqref{3.aux4}, we need
some preparations. We write
$$
  A^{BD}(\bm{c}) = P_L(A(\bm{c})+P_{L^\perp})^{-1}
	= \frac{P_L\operatorname{adj}(A(\bm{c})+P_{L^\perp})}{\det
	(A(\bm{c})+P_{L^\perp})} =: \frac{R(\sqrt{\bm{c}})}{S(\sqrt{\bm{c}})},
$$
where ``adj'' denotes the adjugate matrix. We know that the elements of
$A(\bm{c})$, $P_L$, and $P_{L^\perp}$ are polynomials of $\sqrt{\bm{c}}$.
Therefore, $R(\sqrt{\bm{c}})$ and $S(\sqrt{\bm{c}})$ are also polynomials
of $\sqrt{\bm{c}}$. 
Any eigenvalue of $A(\bm{c})$ is also an eigenvalue of $A(\bm{c})+P_{L^\perp}$
(since $L^\perp=\ker A(\bm{c}) $). As $A(\bm{c})$ has the eigenvalue 0 with
eigenvector $\sqrt{\bm{c}}$, $A(\bm{c})+P_{L^\perp}$ has the eigenvalue 1
with the same eigenvector. Moreover, all other eigenvalues of $A(\bm{c})+P_{L^\perp}$
are larger than or equal to $\mu$. Since the determinant of a matrix
is the product of its eigenvalues, we conclude that
$S(\sqrt{\bm{c}})\ge \mu^{n-1}>0$. This shows that $S(\sqrt{\bm{c}})$ is uniformly
bounded from below. Thus, we can estimate as follows, denoting the elements
of the matrix $R(\sqrt{\bm{c}})$ by $R_{ij}(\sqrt{\bm{c}})$:
\begin{align*}
  \bigg|\frac{\sqrt{c_i}}{\sqrt{\bar{c}_i}} & A_{ij}^{BD}(\bar{\bm{c}})
	- A_{ij}^{BD}(\bm{c})\frac{\sqrt{c_j}}{\sqrt{\bar{c}_j}}\bigg|
	= \bigg|\frac{\sqrt{c_i}R_{ij}(\sqrt{\bar{\bm{c}}})}{\sqrt{\bar{c}_i}
	S(\sqrt{\bar{\bm{c}}})} - \frac{R_{ij}(\sqrt{{\bm{c}}})
	\sqrt{c_j}}{S(\sqrt{{\bm{c}}})\sqrt{\bar{c}_j}}\bigg| \\
	&= \frac{1}{S(\sqrt{\bm{c}})S(\sqrt{\bar{\bm{c}}})\sqrt{\bar{c}_i{\bar{c}_j}}}
	\big|\big(\sqrt{c_i}S(\sqrt{{\bm{c}}}) -\sqrt{\bar{c}_i}S(\sqrt{\bar{\bm{c}}})\big)
	R_{ij}(\sqrt{\bar{\bm{c}}})\sqrt{\bar{c}_j} \nonumber \\
  &\phantom{xx}{}- \big(R_{ij}(\sqrt{\bm{c}})\sqrt{c_j} - R_{ij}(\sqrt{\bar{\bm{c}}})
	\sqrt{\bar{c}_j}\big)\sqrt{\bar{c}_i}S(\sqrt{\bar{\bm{c}}})\big| \nonumber \\
	&\le C(m)\sum_{i=1}^n|\sqrt{c_i}-\sqrt{\bar{c}_i}|, \nonumber 
\end{align*}
where $C(m)>0$ depends on the Lipschitz constants of the polynomials
$\sqrt{c_i}R_{ij}(\sqrt{\bm{c}})$ and $\sqrt{c_i}S(\sqrt{\bm{c}})$.
Inserting this estimate into \eqref{3.aux4}, we obtain
\begin{align}\label{3.T2}
  4\sum_{i,j=1}^n&\int_0^t\int_\Omega
	Z_i\cdot\na\sqrt{\bar{c}_j}\bigg(\frac{\sqrt{c_i}}{\sqrt{\bar{c}_i}}
	A_{ij}^{BD}(\bar{\bm{c}})
	- A_{ij}^{BD}({\bm{c}})\frac{\sqrt{c_j}}{\sqrt{\bar{c}_j}}\bigg)dxds \\
	&\le \lambda\sum_{i=1}^n\int_0^t\int_\Omega|\na(\sqrt{c_i}-\sqrt{\bar{c}_i})|^2dxds
	+ C\sum_{i=1}^n\int_0^t\int_\Omega(\sqrt{c_i}-\sqrt{\bar{c}_i})^2dxds, \nonumber
\end{align}
where $C>0$ also depends on the $L^\infty$ norm of $\na\sqrt{\bar{c}_i}$
through \eqref{3.Z1}.

Finally, we use estimates \eqref{3.T1} and \eqref{3.T2} in the relative entropy 
inequality \eqref{3.Hrel2}, together with Lemma \ref{lem.ent} in Appendix 
\ref{app.ent}, to find that
\begin{align*}
  H(&\bm{c}(t)|\bar{\bm{c}}(t)) + \lambda\sum_{i=1}^n\int_0^t\int_\Omega
	|\na(\sqrt{c_i}-\sqrt{\bar{c}_i})|^2 dxds \\
	&\le C \sum_{i=1}^n\int_0^t\int_\Omega(\sqrt{c_i}-\sqrt{\bar{c}_i})^2dxds
\le C\sum_{i=1}^n\int_0^t H(\bm{c}|\bar{\bm{c}})ds,
\end{align*}
and an application of Gronwall's lemma finishes the proof.

\begin{remark}[Nonhomogeneous total mass]\rm 
The condition $\sum_{i=1}^n c_{i}^0(x)=1$ for $x\in\Omega$ on the initial total mass 
can be relaxed to $\sum_{i=1}^n c_{i}^0(x)=M(x)$ for $x\in\Omega$ and some 
strictly positive
function $M\in L^\infty(\Omega)$. In this situation, the force-flux relations in
\eqref{1.eq} change to
$$
  \na c_i - \frac{c_i}{\sum_{j=1}^n c_j}\sum_{j=1}^n\na c_j
	= -\sum_{j=1,\,j\neq i}^n \frac{c_ic_j}{D_{ij}}(u_i-u_j), \quad i=1,\ldots,n.
$$
Notice that the total mass $\sum_{j=1}^n c_j=M$ is preserved in time. The
previous equation can be expressed in terms of the matrix $A(\bm{c})$,
defined in \eqref{1.A}, by
$$
  \sum_{j=1}^n (P_L)_{ij}\na\sqrt{c_j} = -\sum_{j=1}^n A_{ij}(\bm{c})\sqrt{c_i}u_j,
$$
where the projection matrix $P_L$ is now given by
$(P_L)_{ij}=\delta_{ij}-\sqrt{c_ic_j}/M(x)$, $i,j=1,\ldots,n$.
Lemma \ref{lem.A} still holds in this situation with 
$\alpha\le\inf_{x\in\Omega} M(x)\min_{i\neq j}(1/D_{ij})$. The relative entropy
inequalities \eqref{3.Hrel} and \eqref{3.Hrel2} do not depend on the
assumption $\sum_{i=1}^n c_i=1$ such that the relative entropy inequalities
in Theorem \ref{thm.ws} and Corollary \ref{coro} still hold but with constants
depending on $M$.
\qed
\end{remark}


\section{Weak-strong uniqueness for generalized Maxwell--Stefan systems}\label{sec.gen}

We consider the generalized Maxwell--Stefan system \eqref{1.gms1}--\eqref{1.gms2}.
First, we rewrite \eqref{1.gms2} in terms of the Bott--Duffin inverse of $B(\bm{c})$.
To this end, we recall the definition of $(P_L)_{ij}=\delta_{ij}-\sqrt{c_ic_j}$
and rewrite the right-hand side of \eqref{1.gms2},
\begin{align*}
  c_i\na h'_i(c_i) - c_i\sum_{j=1}^n c_j\na h'_j(c_j)
	&= \sqrt{c_i}\bigg(\sqrt{c_i}\na h'_i(c_i) - \sum_{j=1}^n\sqrt{c_ic_j}
	\sqrt{c_j}\na h'_j(c_j)\bigg) \\
	&= \sqrt{c_i}\sum_{j=1}^n(P_L)_{ij}\sqrt{c_j}\na h'_j(c_j),
\end{align*}
as well as the left-hand side of \eqref{1.gms2}, using definition \eqref{1.AK}
of $B(\bm{c})$,
$$
  -\sum_{j=1}^n K_{ij}(\bm{c})c_ju_j = -\sqrt{c_i}\sum_{j=1}^n B_{ij}(\bm{c})
	\sqrt{c_j}u_j, \quad i=1,\ldots,n,
$$
showing that \eqref{1.gms2} is equivalent to
$$
  -\sum_{j=1}^n B_{ij}(\bm{c})\sqrt{c_j}u_j
	= \sum_{j=1}^n(P_L)_{ij}\sqrt{c_j}\na h'_j(c_j), \quad i=1,\ldots,n.
$$
We prove in Lemma \ref{lem.AK} below that the Bott--Duffin inverse $B^{BD}(\bm{c})$
of $B(\bm{c})$ exists. Thus, we can invert the previous system:
$$
  \sqrt{c_i}u_i = -\sum_{j=1}^n (B^{BD}(\bm{c})P_L)_{ij}\sqrt{c_j}\na h_j'(c_j)
	= -\sum_{j=1}^n B_{ij}^{BD}(\bm{c})\sqrt{c_j}\na h_j'(c_j),
$$
where we have used the relation $B^{BD}(\bm{c})P_L=B^{BD}(\bm{c})$ (see 
\eqref{app.bd} in Appendix \ref{app}). This equation generalizes \eqref{1.eqABD}.
We conclude that system \eqref{1.gms1}--\eqref{1.gms2} can be written as
\begin{equation}\label{4.eqABD}
  \pa_t c_i = \diver\bigg(\sum_{j=1}^n \sqrt{c_i} B_{ij}^{BD}(\bm{c})\sqrt{c_j}
	\na h'_j(c_j)\bigg), \quad i=1,\ldots,n.
\end{equation}


\subsection{Properties of the matrix $B(\bm{c})$}

We prove the following lemma.

\begin{lemma}\label{lem.AK}
Let Assumptions (B1)--(B4) hold for $B(\bm{c})$, defined in \eqref{1.AK}.
Then the Bott--Duffin inverse $B^{BD}(\bm{c})=P_L(B(\bm{c})P_L+P_{L^\perp})^{-1}$ 
of $B(\bm{c})$ exists, is symmetric and satisfies the following properties:
\begin{itemize}
\item Let $s>0$. Then the elements $B_{ij}^{BD}(\bm{c})$ are bounded and Lipschitz 
continuous for all $\bm{c}\in[s,1]^n$.
\item Let $m>0$. Then there exists $\lambda(m)>0$ such that for all
$\bm{z}\in\R^n$ and $\bm{c}\in[m,1]^n$,
\begin{equation}\label{4.ABDpd}
  \bm{z}^T B^{BD}(\bm{c})\bm{z} \ge \lambda(m)|P_L\bm{z}|^2.
\end{equation}
\item The matrix $B^{BD}(\mathbf{c})$ satisfies for all
$\bm{z}\in\R^n$ and $\bm{c}\in[0,1]^n$, 
\begin{equation}\label{4.ABDless}
 \bm{z}^T B^{BD}(\bm{c})\bm{z} \le \frac{1}{\mu}|\bm{z}|^2,
\end{equation}
recalling that $\mu>0$ is a lower bound for the nonzero eigenvalues of
$B(\bm{c})$; see Assumption (B4).
\end{itemize}
\end{lemma}

\begin{proof}
Assumption (B1) and Lemma \ref{lem.bott1} (ii) in Appendix \ref{app} imply that
$$
  \ker(B(\bm{c})P_L+P_{L^\perp}) = \ker(B(\bm{c})P_L)\cap L 
	= L^\perp\cap L = \{0\}.
$$
Hence, $B(\bm{c})P_L+P_{L^\perp}$ is invertible and the Bott--Duffin inverse
is well defined and symmetric. 

We continue by studying the eigenvalues of $B(\bm{c})P_L+P_{L^\perp}$.
A computation shows that for 
$\sqrt{\bm{c}}\in\ker(B(\bm{c})P_L)=L^\perp$ we have
$(B(\bm{c})P_L+P_{L^\perp})\sqrt{\bm{c}}=P_{L^\perp}\sqrt{\bm{c}}=\sqrt{\bm{c}}$,
i.e., $\sqrt{\bm{c}}$ is an eigenvector of $B(\bm{c})P_L+P_{L^\perp}$ with
eigenvalue 1. Let $\xi \notin L^\perp$, $\xi \ne 0$, be another eigenvector,
$$
  ( B(\bm{c})P_L+P_{L^\perp} ) \xi = \rho \xi.
$$
Then $\rho \ne 0$. Applying $P_L$ on both sides, we obtain 
$P_L B(\bm{c})P_L \xi = B(\bm{c}) (P_L \xi ) = \rho P_L \xi$, i.e., 
$P_L \xi \ne 0$ is an eigenvector of $B(\bm{c})$ with eigenvalue $\rho$. 
Due to Assumption (B4), we conclude that $\rho \ge \mu > 0$.

We claim that the elements $B_{ij}^{BD}(\bm{c})$ are bounded and Lipschitz 
continuous for all $\bm{c}\in[s,1]^n$. Indeed, observe that
\begin{equation}\label{4.adj}
  B^{BD}(\bm{c}) = P_L(B(\bm{c})P_L+P_{L^\perp})^{-1}
	= \frac{P_L\operatorname{adj}(B(\bm{c})P_L+P_{L^\perp})}{
	\det(B(\bm{c})P_L+P_{L^\perp})},
\end{equation}
where ``adj'' denotes the adjugate.
Since the determinant of a matrix is the product of its eigenvalues,
$\det(B(\bm{c})P_L+P_{L^\perp})\ge\mu^{n-1}>0$ and the denominator in
\eqref{4.adj} is bounded from below.
Assumption (B2) implies that the elements of $B(\bm{c})$ are bounded. 
Hence, all elements
of $\operatorname{adj}(B(\bm{c})P_L+P_{L^\perp})$ are bounded too. We conclude 
from \eqref{4.adj} that the elements of $B^{BD}(\bm{c})$ are bounded.
Since the product of Lipschitz continuous functions is Lipschitz continuous,
Assumption (B2) further implies that the elements of $B^{BD}(\bm{c})$ are
Lipschitz continuous for all $\bm{c}\in[s,1]^n$ for any $s>0$.

We wish to verify \eqref{4.ABDpd}. Since the spectral radius $r$ of a matrix
is bounded by its Frobenius norm $\|\cdot\|_F$ and the Frobenius norm is
submultiplicative, we have
\begin{align*}
  r(B(\bm{c})P_L+P_{L^\perp}) &\le \|B(\bm{c})P_L+P_{L^\perp}\|_F
	\le \|B(\bm{c})P_L\|_F + \|P_{L^\perp}\|_F \\
	&\le \|B(\bm{c})\|_F\|P_L\|_F + \|P_{L^\perp}\|_F.
\end{align*}
The Frobenius norms of $P_L$ and $P_{L^\perp}$ are estimated according to
\begin{align*}
  \|P_L\|_F^2 &= \sum_{i,j=1}^n(\delta_{ij}-\sqrt{c_ic_j})^2
	\le \sum_{i,j=1}^n 1 = n^2, \\
	\|P_{L^\perp}\|_F^2 &= \sum_{i,j=1}^n(\sqrt{c_ic_j})^2 
	= \bigg(\sum_{i=1}^n c_i\bigg)^2 = 1.
\end{align*}
Assumption (B3) guarantees that $\|B(\bm{c})\|_F\le \gamma(m)$.
This shows that $r(B(\bm{c})P_L+P_{L^\perp})\le {\gamma}(m)n+1$. We infer that
the smallest eigenvalue of $(B(\bm{c})P_L+P_{L^\perp})^{-1}$ is larger than
$ \lambda(m) := 1/({\gamma}(m)n+1)$ proving  \eqref{4.ABDpd}.

It remains to prove \eqref{4.ABDless}. First,
we show that the nonzero eigenvalues of $B(\bm{c})$ and $B^{BD}(\bm{c})$ 
are reciprocal to each other. 
Let $\ell\in\R$ be a nonzero eigenvalue of $B^{BD}(\bm{c})$. 
Then the corresponding eigenvector $\bm{y} \in L$ satisfies
$B^{BD}(\bm{c})\bm{y} = \ell \bm{y}$, which is 
$P_L (B(\bm{c})P_L + P_{L^\perp})^{-1}\bm{y} = \ell \bm{y}$. 
Hence, $\bm{z}:=(B(\bm{c})P_L + P_{L^\perp})^{-1}\bm{y}$ satisfies 
$P_L \bm{z} = \ell (B(\bm{c})P_L+P_{L^\perp})\bm{z}$. Applying $P_L$ on both sides 
yields $P_L \bm{z} = \ell P_L B(\bm{c}) P_L \bm{z} = \ell B(\bm{c})P_L \bm{z}$. 
Thus, $P_L \bm{z}$ is an eigenvector of $B(\mathbf{c})$ with eigenvalue $1/\ell$. 
Similarly, we can reverse the above argument and verify that if $\bm{z}$ is a 
nonzero eigenvector of $B(\bm{c})$ with eigenvalue $\ell$, then 
$(B(\bm{c})P_L+P_{L^\perp}) \bm{z}$ is an eigenvector of $B^{BD}(\bm{c})$ 
with eigenvalue $1/\ell$. We conclude that the largest eigenvalue of 
$B^{BD}(\bm{c})$ is the reciprocal of the smallest eigenvalue of $B(\bm{c})$, 
and  Assumption (B4) implies that 
$\bm{z}^T B^{BD}(\bm{c})\bm{z} \le  |\bm{z}|^2/\mu$ for all $\bm{z}\in\R^n$.
\end{proof}


\subsection{Weak and strong solutions}

We call $\bm{c}$ a {\em weak solution} to \eqref{1.bic}, 
\eqref{1.gms1}--\eqref{1.gms2} if
$$
  c_i\in C_{\rm loc}^0([0,\infty);{\mathcal V}')\cap 
	L_{\rm loc}^2(0,\infty;H^1(\Omega)), \quad i=1,\ldots,n,
$$
where ${\mathcal V}'$ is the dual space of ${\mathcal V}=\{w\in H^2(\Omega):
\na w\cdot\nu=0$ on $\pa\Omega\}$, it holds for any test function
$\phi_i\in C_{\rm loc}^1([0,\infty);C^1(\overline\Omega))$ with $\na\phi_i\cdot\nu=0$
on $\pa\Omega$ and all $t>0$ that
\begin{align*}
  \int_\Omega & c_i(t)\phi_i(t)dx - \int_\Omega c_{i}^0\phi_i(0)dx
	- \int_0^t\int_\Omega c_i\pa_t\phi_i dxds \\
	&{}+ \sum_{j=1}^n\int_0^t\int_\Omega\sqrt{c_i}B_{ij}^{BD}(\bm{c})\sqrt{c_j}
	\na h'_j(c_j)\cdot\na\phi_i dxds = 0, \quad i=1,\ldots,n, 
\end{align*}
and the entropy dissipation inequality
\begin{equation}\label{4.ei}
  H(\bm{c}(t)) + \sum_{i,j=1}^n\int_0^t\int_\Omega \sqrt{c_ic_j}B_{ij}^{BD}(\bm{c})
	\na h'_i(c_i)\cdot\na h_j'(c_j)dxds \le H(\bm{c}^0)
\end{equation}
is satisfied for $t>0$.

Furthermore, we call $\bar{\bm{c}}$ a strong solution to \eqref{1.bic}, 
\eqref{1.gms1}--\eqref{1.gms2} 
if $c_i\in C_{\rm loc}^1([0,\infty);C^1(\overline\Omega))$ for $i=1,\ldots,n$,
if \eqref{1.bic}, \eqref{1.gms1}--\eqref{1.gms2} are satisfied pointwise, and 
the entropy identity \eqref{1.entid} is fulfilled. 


\subsection{Relative entropy inequality}

The partial free energies $h_i (c_i)$ and pressures $p_i (c_i)$ 
are associated through \eqref{4.defpress}.
Define the associated relative free energy density and relative pressure via
\begin{equation}\label{4.relfundef}
\begin{aligned}
  h_i(c_i|\bar{c}_i) &:= h_i(c_i) - h_i(\bar{c}_i) - h_i'(\bar{c}_i)(c_i-\bar{c}_i), \\
  p_i (c_i|\bar{c}_i) &:= p_i(c_i) - p_i(\bar{c}_i) - p_i'(\bar{c}_i)(c_i-\bar{c}_i).
\end{aligned}
\end{equation}
We prove a relative entropy inequality associated to the generalized
Maxwell--Stefan system.

\begin{lemma}\label{lem.relent3}
Let $\bm{c}$ be a weak solution to \eqref{1.bic}, \eqref{1.gms1}--\eqref{1.gms2}
and let $\bar{\bm{c}}$ be a strong solution to 
\eqref{1.bic}, \eqref{1.gms1}--\eqref{1.gms2} satisfying
$$
  h'_i(\bar{c}_i)\in L_{\rm loc}^2(0,\infty;H^2(\Omega)), \
	h''_i(\bar{c}_i)\in L_{\rm loc}^\infty(0,\infty;L^\infty(\Omega)), \
	\pa_t\bar{c}_i\in L_{\rm loc}^2(0,\infty;L^2(\Omega)).
$$
Then the following relative entropy inequality holds:
\begin{align}
  H(&\bm{c}(t)|\bar{\bm{c}}(t)) - H(\bm{c}^0|\bar{\bm{c}}^0) 
	+ \sum_{i,j=1}^n\int_0^t\int_\Omega B_{ij}^{BD}(\bm{c})Y_i\cdot Y_j dxds
  \nonumber \\
  &\le -\sum_{i,j=1}^n\int_0^t\int_\Omega\bigg(B_{ij}^{BD}(\bm{c})\sqrt{c_j}
	- \frac{\sqrt{c_i}B_{ij}^{BD}(\bar{\bm{c}})\sqrt{\bar{c}_j}}{\sqrt{\bar{c}_i}}\bigg)
	Y_j\cdot\na h'_j(\bar{c}_j) dxds \nonumber \\
  &\phantom{xx}{}+ \sum_{i=1}^n\int_0^t\int_\Omega p_i (c_i|\bar{c}_i) 
	\diver\bigg(\sum_{j=1}^n B_{ij}^{BD}(\bar{\bm{c}}) 
	\frac{\sqrt{\bar{c}_j}}{\sqrt{\bar{c}_i}}\na h'_j(\bar{c}_j)\bigg)dxds,
  \label{4.relenineq}
\end{align}
where
\begin{equation}\label{4.defY}
  Y_i = \sqrt{c_i} \na(h'_i(c_i)-h'_i(\bar{c}_i)),  \quad i = 1,\ldots,n .
\end{equation}
\end{lemma}

Note that the definition for $Y_i$ differs 
from that one used in Section \ref{sec.ws}.

\begin{proof}
We proceed as in the proof of Lemma \ref{lem.relent2}, but re-arrange the terms
in a different fashion. 
The difference $c_i-\bar{c}_i$ satisfies the weak formulation
\begin{align*}
  0 &= \int_\Omega (c_i-\bar{c}_i)(t)\phi_i(t)dx 
	- \int_\Omega (c_{i}^0-\bar{c}_{i}^0)\phi_i(0)dx
	- \int_0^t\int_\Omega (c_i-\bar{c}_i) \pa_t\phi_i dxds \\
	&\phantom{xx}{}
	+ \sum_{j=1}^n\int_0^t\int_\Omega\big(\sqrt{c_i}B_{ij}^{BD}(\bm{c})\sqrt{c_j}
	\na h'_j(c_j) -	\sqrt{\bar{c}_i}B_{ij}^{BD}(\bar{\bm{c}})\sqrt{\bar{c}_j}
	\na h'_j(\bar{c}_j)\big)\cdot\na\phi_i dxds
\end{align*}
for $i=1,\ldots,n$.
We wish to use $\phi_i=h'_i(\bar{c}_i)$ as a test function. Strictly speaking,
this is not possible, but, as in the proof of Lemma \ref{lem.relent}, we can
use a density argument. Then, using \eqref{4.eqABD} for the third term
and adding over $i=1,\ldots,n$, we obtain
\begin{equation}
\label{4.eqdiff}
\begin{aligned}
  0 &= \sum_{i=1}^n\int_\Omega (c_i-\bar{c}_i)(t)h'_i(\bar{c}_i(t))dx 
	- \sum_{i=1}^n\int_\Omega (c_{i}^0-\bar{c}_{i}^0)h'_i(\bar{c}_{i0})dx \\
	&\phantom{xx}{}- \sum_{i=1}^n\int_0^t\int_\Omega (c_i-\bar{c}_i)h''_i(c_i)
	\diver\bigg(\sum_{j=1}^n\sqrt{\bar{c}_i}B_{ij}^{BD}(\bar{\bm{c}})\sqrt{\bar{c}_j}
	\na h'_j(\bar{c}_j)\bigg)dxds \\
	&\phantom{xx}{}
	+ \sum_{i,j=1}^n\int_0^t\int_\Omega\big(\sqrt{c_i}B_{ij}^{BD}(\bm{c})\sqrt{c_j}
	\na h'_j(c_j) -	\sqrt{\bar{c}_i}B_{ij}^{BD}(\bar{\bm{c}})\sqrt{\bar{c}_j}
	\na h'_j(\bar{c}_j)\big)\cdot\na h'_i(\bar{c}_i) dxds.
\end{aligned}
\end{equation}
Subtracting \eqref{4.eqdiff} and the entropy identity \eqref{1.entid}
for $\bar{\bm{c}}$ from the entropy inequality \eqref{4.ei} for $\bm{c}$,
we find that
\begin{align}
  H(\bm{c}(t)|\bar{\bm{c}}(t)) &\le H(\bm{c}^0|\bar{\bm{c}}^0) 
	- \sum_{i,j=1}^n\int_0^t\int_\Omega B_{ij}^{BD}(\bm{c})
	\sqrt{c_ic_j}\na(h'_i(c_i)-h'_i(\bar{c}_i))\cdot\na h'_j(c_j)dxds \nonumber \\
	&\phantom{xx}{}- \sum_{i=1}^n\int_0^t\int_\Omega(c_i-\bar{c}_i)h_i''(\bar{c}_i)
	\diver\bigg(\sum_{j=1}^n \sqrt{\bar{c}_i}B_{ij}^{BD}(\bar{\bm{c}})\sqrt{\bar{c}_j}
	\na h'_j(\bar{c}_j)\bigg)dxds. \label{4.relenineq2}
\end{align}
In turn, using \eqref{4.defY}, this is rewritten as
\begin{align*}
  H(&\bm{c}(t)|\bar{\bm{c}}(t)) - H(\bm{c}^0|\bar{\bm{c}}^0) 
  + \sum_{i,j=1}^n\int_0^t\int_\Omega B_{ij}^{BD}(\bm{c}) Y_i \cdot Y_j dxds \\
  &\le -\sum_{i,j=1}^n\int_0^t \int_\Omega\bigg(B_{ij}^{BD}(\bm{c})\sqrt{c_j}
	- \sqrt{c_i}B_{ij}^{BD}(\bar{\bm{c}}) \frac{\sqrt{\bar{c}_j}}{\sqrt{\bar{c}_i}}\bigg)
	Y_i\cdot\na h'_j(\bar{c}_j) dxds \\
  &\phantom{xx}{} -\sum_{i,j=1}^n\int_0^t\int_\Omega\sqrt{c_i}B_{ij}^{BD}(\bar{\bm{c}}) 
  \frac{\sqrt{\bar{c}_j}}{\sqrt{\bar{c}_i}}Y_i\cdot\na h'_j(\bar{c}_j) dxds \\
  &\phantom{xx}{}-\int_0^t\int_\Omega \sum_{i=1}^n  (c_i-\bar{c}_i)
	\bar{c}_i h_i''(\bar{c}_i)\diver\bigg(\sum_{j=1}^n B_{ij}^{BD}(\bar{\bm{c}}) 
	\frac{\sqrt{\bar{c}_j}}{\sqrt{\bar{c}_i}}\na h'_j(\bar{c}_j)\bigg)dxds \\
  &\phantom{xx}{} -\int_0^t\int_\Omega \sum_{i=1}^n (c_i-\bar{c}_i)
	h_i''(\bar{c}_i) \na \bar{c}_i\cdot \sum_{j=1}^n B_{ij}^{BD}(\bar{\bm{c}}) 
	\frac{\sqrt{\bar{c}_j}}{\sqrt{\bar{c}_i}}\na h'_j(\bar{c}_j) dxds \\
  &=: J_1 + J_2 + J_3 + J_4.
\end{align*}
The sum $J_2 + J_4$ becomes
\begin{align*}
  J_2 + J_4 &= -\int_0^t\int_\Omega \sum_{i=1}^n \big(c_i\na h_i'(c_i) 
	- \bar{c}_i\na h_i'(\bar{c}_i)\big) 
	\cdot \sum_{j=1}^n B_{ij}^{BD}(\bar{\bm{c}})
	\frac{\sqrt{\bar{c}_j}}{\sqrt{\bar{c}_i}}\na h'_j(\bar{c}_j) dxds \\
  &= \int_0^t\int_\Omega \sum_{i=1}^n(p_i(c_i) - p_i(\bar{c}_i)) 
  \diver\bigg(\sum_{j=1}^n B_{ij}^{BD}(\bar{\bm{c}}) 
	\frac{\sqrt{\bar{c}_j}}{\sqrt{\bar{c}_i}}\na h'_j(\bar{c}_j)\bigg)dxds.
\end{align*}
Combining this expression with $J_3$ and using definition \eqref{4.relfundef}
finally leads to \eqref{4.relenineq}.
\end{proof}


\subsection{The entropy dissipation structure} 

We state an auxiliary lemma that provides some control of the entropy inequality 
\eqref{4.ei} and the relative entropy inequality \eqref{4.relenineq}. 

\begin{lemma}\label{lem.dis}
Let $\bm{c}$ be a weak solution and $\bar{\bm{c}}$ be a strong solution 
to \eqref{1.bic}, \eqref{1.gms1}--\eqref{1.gms2}, satisfying the hypotheses 
of Lemma \ref{lem.relent3} and $\bar{c}_i(t)\ge m$ in $\Omega$, 
$t>0$ for some constant $m>0$.

\begin{itemize}
\item[(i)] Assume that $c_i \ge m/2$ for all $i = 1\ldots,n$ and let 
$Z_i = \sqrt{c_i}\na h_i'(c_i)$. Then, for some constant $\beta(m) > 0$, we have
\begin{equation}\label{4.esdis1}
  \sum_{i,j = 1}^n B_{ij}^{BD}(\bm{c})Z_i\cdot Z_j  
	\ge 2\beta(m)\sum_{i=1}^n|\na c_i|^2.
\end{equation}
	
\item[(ii)] Assume that ${c}_i \ge m/2$ for all $i = 1,\ldots,n$ and let  
$Y_i = \sqrt{c_i}\na(h'_i(c_i)-h'_i(\bar{c}_i))$. Then, 
for some $\beta(m) > 0$ and $C > 0$, we have
\begin{equation} \label{4.esdis3}
  \sum_{i,j = 1}^n B_{ij}^{BD}(\bm{c})Y_i\cdot Y_j  
	\ge \beta(m)\sum_{i=1}^n |\na c_i - \na \bar{c}_i |^2 
	- C\sum_{i=1}^n |c_i - \bar{c}_i|^2.
\end{equation}

\item[(iii)] Let the weak solution $\bm{c}$ satisfy \eqref{4.ei} and set
\begin{equation}\label{4.defcstar}
  c_* (x,t) := \min_{i = 1,\ldots,n} c_i (x,t).
\end{equation}
Then
\begin{equation}\label{4.esdis2}
  H(\bm{c}(t)) +  2\beta(m) \int_0^t\int_\Omega   
	\mathrm{1}_{\{c_* >  m/2\}}\sum_{i=1}^n |\na c_i|^2dxds \le H(\bm{c}^0).
\end{equation}
\end{itemize}
\end{lemma}

Note that \eqref{4.esdis2} provides a partial control of the gradients, 
which however might degenerate as $m$ tends to zero.

\begin{proof}
	\emph{Proof of (i).} 
	Inequality \eqref{4.ABDpd} in Lemma \ref{lem.AK} implies that 
	\begin{align}
	  \sum_{i,j=1}^n B_{ij}^{BD}(\bm{c})Z_i\cdot Z_j
		&\ge \lambda(m/2) |P_L\bm{Z}|^2 = \lambda( m/2) \bm{Z}^TP_L^TP_L\bm{Z} 
		\label{4.aux1} \\
		&= \lambda( m/2) \sum_{i,j=1}^n(P_L)_{ij}Z_i\cdot Z_j. \nonumber
	\end{align}
	Before we can estimate the right-hand side, we need some preparations.
	
	We define the vector $\widetilde{\bm{c}}:=(c_1,\ldots,c_{n-1})$ 
	without the last component and define the entropy density in $n-1$ variables
	according to
	$$
	  \widetilde{h}(\widetilde{\bm{c}}) = \sum_{i=1}^{n-1}h_i(c_i)
		+ h_n\bigg(1-\sum_{j=1}^{n-1}c_j\bigg).
	$$
	Its partial derivative is given by
	$$
	  \widetilde{h}_i'(\widetilde{\bm{c}}) := \frac{\pa\widetilde{h}(\widetilde{\bm{c}})}{
		\pa c_i} = h_i'(c_i) - h'_n\bigg(1-\sum_{j=1}^{n-1}c_j\bigg), \quad i=1,\ldots,n-1.
	$$
	Next, introduce the matrix $E(\bm{c})$ with elements 
	$$
	E_{ij}(\bm{c})
	=(P_L)_{ij}\sqrt{c_ic_j} = c_i\delta_{ij}-c_ic_j 
	=	\begin{cases}
	c_i - c_i^2  &\quad\mbox{if } i=j. \\
	- c_i c_j    &\quad\mbox{if } i \neq j. \\
	\end{cases}
	$$
	The sum of its rows and columns vanishes, $\sum_{j=1}^n E_{ij}(\bm{c})
	= \sum_{i=1}^n E_{ij}(\bm{c}) = 0$. We deduce from
  the symmetry of $E(\bm{c})$ that for all $z_i\in\R^d$,
	\begin{align*}
	  \sum_{i,j=1}^n E_{ij}(\bm{c})z_i\cdot z_j 
		&= \sum_{i=1}^n z_i\cdot\bigg(\sum_{j=1}^{n-1}E_{ij}(\bm{c})z_j 
		+ E_{in}(\bm{c})z_n\bigg) 
		= \sum_{i=1}^n z_i\cdot\sum_{j=1}^{n-1}E_{ij}(\bm{c})(z_j-z_n)  \\
		&= \sum_{i=1}^{n-1} z_i\cdot\sum_{j=1}^{n-1}E_{ij}(\bm{c})(z_j-z_n)
		+ z_n\cdot\sum_{j=1}^{n-1}E_{nj}(\bm{c})(z_j-z_n) \\
		&= \sum_{i,j=1}^{n-1}E_{ij}(\bm{c})(z_i-z_n)\cdot(z_j-z_n).
	\end{align*}
	Choosing $z_i=\na h'_i(c_i)$ and observing that
	$Z_i=\sqrt{c_i}z_i$, we rewrite the right-hand side of \eqref{4.aux1}:
	\begin{align*}
	  \sum_{i,j=1}^n(P_L)_{ij}Z_i\cdot Z_j
		&= \sum_{i,j=1}^{n-1} E_{ij}(\bm{c}) (z_i-z_n)\cdot (z_j-z_n) \\ 
		&= \sum_{i,j=1}^{n-1}E_{ij}(\bm{c})
		\na \widetilde{h}_i'(\widetilde{\bm{c}})
		\cdot
		\na\widetilde{h}_j'(\widetilde{\bm{c}}).
	\end{align*}
	Introducing the matrix $Q(\widetilde{\bm{c}})$ with  elements
	$Q_{ij}(\widetilde{\bm{c}}) = \pa^2 \widetilde{h} (\widetilde{\bm{c}})/\pa c_i\pa c_j$
	for $i,j=1,\ldots,n-1$, this expression becomes
	\begin{align}
	  \sum_{i,j=1}^n&(P_L)_{ij}Z_i\cdot Z_j
		= \sum_{i,j,k,\ell=1}^{n-1}E_{ij}(\bm{c})
		Q_{ik}(\widetilde{\bm{c}})\na c_k 
		\cdot Q_{j\ell}(\widetilde{\bm{c}})\na c_\ell. \label{4.aux2}
	\end{align}
	We claim that there exists $\zeta(m)>0$ such that for all $\bm{y}\in\R^{n-1}$,
	\begin{equation}\label{4.QEQ}
	  \bm{y}^T\big(Q(\widetilde{\bm{c}})^T E(\bm{c})Q(\widetilde{\bm{c}})\big)\bm{y}
		\ge \zeta(m)|\bm{y}|^2.
	\end{equation}
	Then, letting $\bm{y} = \na\bm{c}$ in \eqref{4.QEQ} and using \eqref{4.aux1} 
	and \eqref{4.aux2} leads to \eqref{4.esdis1} 
	with $\beta(m)=\zeta(m)\lambda(m/2)/2$.
	The proof of \eqref{4.QEQ} proceeds in several steps.
	
	Consider first the matrix  $Q (\widetilde{\bm{c}})$. 
	Let $\eta:=\min_{i=1,\ldots,n}\min_{m/2\le c_i\le 1} h''_i(c_i)>0$ and 
	$\xi \in \mathbb{R}^{n-1}$ and compute
	\begin{align*}
	Q_{ij}(\widetilde{\bm{c}}) 
	&= h_i''(c_i)\delta_{ij} + h_n''\bigg( 1 - \sum_{k=1}^{n-1} c_k \bigg), 
	\\
	\xi^T Q \xi &= \sum_{j=1}^{n-1} h_j''(c_j) \xi_j^2 
	+ h_n''\bigg( 1 - \sum_{k=1}^{n-1} c_k \bigg) 
	(\xi_1 + \cdots + \xi_{n-1})^2 \ge \frac{\eta}{2} |\xi|^2.
	\end{align*}
	This implies that $Q(\widetilde{\bm{c}})$
	is positive definite with eigenvalues larger than or equal to $\eta/2$. 
	Consider next the $(n-1)\times(n-1)$ submatrix 
	$\widetilde{P}_L=((P_L)_{ij})_{i,j=1,\ldots,n-1}$
	of $P_L$ and note that for $\xi \in \R^{n-1}$, we have
	$$
	\xi^T\widetilde{P}_L \xi = |\xi |^2 - (\xi  \cdot \sqrt{\widetilde{c}})^2 
	\ge  |\xi |^2 - |\xi |^2 (c_1 + \cdots + c_{n-1}) 
	= c_n |\xi |^2 \ge \frac{m}{2} |\xi|^2.
	$$
	
	Finally, let $\widetilde{E}(\bm{c})$ be the first $(n-1)\times(n-1)$ submatrix 
	of $E(\bm{c})$. Then $\widetilde{E}(\bm{c})=S^T\widetilde{P}_LS$ with
	$S=\operatorname{diag}(\sqrt{c_1},\ldots,\sqrt{c_{n-1}})$
	and, for all $\bm{y}\in\R^{n-1}$,
	$$
	  \bm{y}^T\widetilde{E}(\bm{c})\bm{y}
		= (S\bm{y})^T\widetilde{P}_L(S\bm{y})
		\ge\frac{m}{2}|S\bm{y}|^2 = \frac{m}{2}\sum_{i=1}^{n-1}c_iy_i^2
		\ge \frac{m^2}{4}\sum_{i=1}^n y_i^2,
	$$
	i.e., the eigenvalues of $\widetilde{E}(\bm{c})$
	are larger than or equal to $m^2/4$. Since $\widetilde{E}(\bm{c})-(m^2/8)I_{n-1}$,
	with $I_{n-1}$ is the unit matrix on $\R^{(n-1)\times(n-1)}$, 
	and $Q(\widetilde{\bm{c}})$ are symmetric and positive definite, we deduce that
	$$
	  Q(\widetilde{\bm{c}})^T   \widetilde{E}(\bm{c}) Q(\widetilde{\bm{c}}) 
		\ge  Q(\widetilde{\bm{c}})^T \frac{m^2}{8}I_{n-1} Q(\widetilde{\bm{c}})
	  \ge \frac{m^2\eta^2}{32}I_{n-1} \, .
	$$
	This proves \eqref{4.QEQ} with $\zeta(m)=m^2\eta^2/32$. 

	\emph{Proof of (ii).} Inequality \eqref{4.ABDpd} in Lemma \ref{lem.AK} implies that 
	\begin{equation}\label{4.aux1-1}
		\sum_{i,j = 1}^n B_{ij}^{BD}(\bm{c}) Y_i \cdot Y_j 
		\ge \lambda(m/2) |P_L \bm{Y}|^2 
		= \lambda(m/2) \sum_{i,j=1}^n (P_L)_{ij} Y_i\cdot Y_j.
	\end{equation}
	Similarly as the derivation of \eqref{4.aux2}, we compute
\begin{align}
  \sum_{i,j=1}^n&(P_L)_{ij}Y_i\cdot Y_j
	= \sum_{i,j,k,\ell=1}^{n-1}E_{ij}(\bm{c})
	\big(Q_{ik}(\widetilde{\bm{c}})\na c_k 
	- Q_{ik}(\widetilde{\bar{\bm{c}}})\na\bar{c}_k\big)
	\cdot\big(Q_{j\ell}(\widetilde{\bm{c}})\na c_\ell 
	- Q_{j\ell}(\widetilde{\bar{\bm{c}}})\na\bar{c}_\ell\big) \nonumber \\
	&= \sum_{i,j,k,\ell=1}^{n-1}E_{ij}(\bm{c})\big(Q_{ik}(\widetilde{\bm{c}})
	\na(c_k-\bar{c}_k) + (Q_{ik}(\widetilde{\bm{c}})-Q_{ik}(\widetilde{\bar{\bm{c}}}))
	\na\bar{c}_k\big) \nonumber \\ 
	&\phantom{xx}{}\times\big(Q_{j\ell}(\widetilde{\bm{c}})	\na(c_\ell-\bar{c}_\ell) 
	+ (Q_{j\ell}(\widetilde{\bm{c}})-Q_{j\ell}(\widetilde{\bar{\bm{c}}}))
	\na\bar{c}_\ell\big). \nonumber
\end{align}
We remark that if $E$ is any symmetric positive definite matrix, then for any 
$\bm{z}_1$, $\bm{z}_2\in\R^n$, the Cauchy--Schwarz and Young's inequalities show that
\begin{align*}
  (\bm{z}_1&+\bm{z}_2)^T E (\bm{z}_1+\bm{z}_2)
	= \bm{z}_1^TE\bm{z}_1 + \bm{z}_1^T E\bm{z}_2 + \bm{z}_2^T E\bm{z}_1
	+ \bm{z}_2^T E\bm{z}_2 \\
	&\ge \bm{z}_1^T E\bm{z}_1 - \sqrt{\bm{z}_1^T E\bm{z}_1}\cdot
	\sqrt{\bm{z}_2^T E\bm{z}_2} - \sqrt{\bm{z}_1^T E\bm{z}_1}\cdot
	\sqrt{\bm{z}_2^T E\bm{z}_2} + \bm{z}_2^T E\bm{z}_2 \\
	&\ge \bm{z}_1^T E\bm{z}_1 - \frac12\bm{z}_1^T E\bm{z}_1 
	- 2\bm{z}_2^T E\bm{z}_2 + \bm{z}_2^T E\bm{z}_2 
	= \frac12\bm{z}_1^T E\bm{z}_1 - \bm{z}_2^T E\bm{z}_2.
\end{align*}
Using this inequality, \eqref{4.aux2} is estimated as
\begin{align}\label{4.I6I7}
  &\sum_{i,j=1}^n(P_L)_{ij}Y_i\cdot Y_j
	\ge \frac12\sum_{i,j,k,\ell=1}^{n-1} E_{ij}(\bm{c})Q_{ik}(\widetilde{\bm{c}})
	Q_{j\ell}(\widetilde{\bm{c}})\na(c_k-\bar{c}_k)\cdot\na(c_\ell-\bar{c}_\ell) \\
	&\phantom{xx}{}- \sum_{i,j,k,\ell=1}^{n-1}E_{ij}(\bm{c})
	\big((Q_{ik}(\widetilde{\bm{c}})-Q_{ik}(\widetilde{\bar{\bm{c}}}))\na\bar{c}_k\big)
	\cdot\big((Q_{j\ell}(\widetilde{\bm{c}})-Q_{j\ell}(\widetilde{\bar{\bm{c}}}))
	\na\bar{c}_\ell\big) =: J_{5}+ J_{6}. \nonumber
\end{align}
	We infer from \eqref{4.QEQ} that
	\begin{equation}\label{4.aux3}
	  J_{5} \ge \frac12 \zeta(m)\sum_{i=1}^{n-1}|\na(c_i-\bar{c}_i)|^2.
	\end{equation}
	It follows from $\na c_n=-\sum_{k=1}^{n-1}\na c_k$ that
	\begin{align*}
	  |\na(c_n-\bar{c}_n)|^2 = \bigg|\sum_{k=1}^{n-1}\na(c_k-\bar{c}_k)\bigg|^2
		&\le (n-1)\sum_{k=1}^{n-1}|\na(c_k-\bar{c}_k)|^2, \\
	  \sum_{k=1}^{n}|\na(c_k-\bar{c}_k)|^2
		&\le n\sum_{k=1}^{n-1}|\na(c_k-\bar{c}_k)|^2.
	\end{align*}
	Inserting these estimates into \eqref{4.aux3} yields finally
	\begin{equation}\label{4.I6}
	  J_{5} \ge \frac{\zeta(m)}{2n}\sum_{i=1}^n|\na(c_i-\bar{c}_i)|^2.
	\end{equation}
	
	The estimate of the term $J_{2}$ is easier. Since $E_{ij}(\bm{c})$ is bounded
	and the Hessian $Q(\widetilde{\bm{c}})=D^2\widetilde{h}$ is Lipschitz continuous,
	$$
	J_6 \le C\sum_{i,k=1}^{n-1}(Q_{ik}(\widetilde{\bm{c}})
		-Q_{ik}(\widetilde{\bar{\bm{c}}}))^2|\na\bar{c}_k|^2
		\le C\sum_{i=1}^{n-1}(c_i-\bar{c}_i)^2|\na\bar{c}_i|^2
		\le C\sum_{i=1}^{n-1}(c_i-\bar{c}_i)^2.
	$$
	Combining the above inequality with \eqref{4.I6}  and \eqref{4.I6I7} gives
	$$
	  \sum_{i,j=1}^n(P_L)_{ij}Y_i\cdot Y_j
		\ge \frac{\zeta(m)}{2n}\sum_{i=1}^n|\na(c_i-\bar{c}_i)|^2
		- C\sum_{i=1}^{n-1}(c_i-\bar{c}_i)^2.
	$$
	We conclude \eqref{4.esdis3} after inserting the previous estimate into
	\eqref{4.aux1-1}.
	
	\emph{Proof of (iii).}  Let $c_*(x,t)$ be defined by \eqref{4.defcstar} and 
	split the domain of integration into the two subdomains
	$$
	  \Omega\times(0,t) = \bigg\{ c_* > \frac{m}{2} \bigg\} \cup 
		\bigg\{ c_* \le  \frac{m}{2} \bigg\}.
	$$ 
	By Lemma \ref{lem.AK}, the matrix $B^{BD}(\bm{c})$ is symmetric and 
	positive semi-definite. Using \eqref{4.esdis1}, the entropy
	inequality \eqref{4.ei} yields \eqref{4.esdis2}.
	\end{proof}

\subsection{Proof of Theorem \ref{thm.gws}}

Lemma \ref{lem.relent3} suggests that the relative entropy inequality can be 
expressed in two ways, using either \eqref{4.relenineq2} or \eqref{4.relenineq}:
\begin{align}
  \nonumber 
  H(\bm{c}|\bar{\bm{c}})(t) &\le H(\bm{c}^0|\bar{\bm{c}}^0)
	+ \int_0^t\int_\Omega(I_1+I_2)dxds \\
  &= H(\bm{c}^0|\bar{\bm{c}}^0)
		+ \int_0^t\int_\Omega(I_3+I_4 + I_5)dxds, \label{4.I3I4I5}
\end{align}
where
\begin{align*}
  I_1 &=  - \sum_{i,j=1}^n B_{ij}^{BD}(\bm{c})
	\sqrt{c_ic_j}\na(h'_i(c_i)-h'_i(\bar{c}_i))\cdot\na h'_j(c_j), \\
	I_2 &=  - \sum_{i=1}^n(c_i-\bar{c}_i)h_i''(\bar{c}_i)
	\diver\bigg(\sum_{j=1}^n \sqrt{\bar{c}_i}B_{ij}^{BD}(\bar{\bm{c}})\sqrt{\bar{c}_j}
	\na h'_j(\bar{c}_j)\bigg), \\
	I_3 &= -\sum_{i,j=1}^n  Y_i \cdot  B^{BD}_{ij}(\bm{c}) Y_j, \\
	I_4 &= -\sum_{i,j=1}^n Y_i \cdot \bigg(B_{ij}^{BD}(\bm{c})\sqrt{c_j}
	 - \frac{\sqrt{c_i}B_{ij}^{BD}(\bar{\bm{c}})\sqrt{\bar{c}_j}}{\sqrt{\bar{c}_i}}\bigg)
	   \na h'_j(\bar{c}_j), \\
	I_5 &= \sum_{i,j=1}^n p_i(c_i|\bar{c}_i)
	\diver\bigg(\frac{B_{ij}^{BD}(\bar{\bm{c}})\sqrt{\bar{c}_j}}{\sqrt{\bar{c}_i}}
	\na h'_j(\bar{c}_j)\bigg), 
\end{align*}
and $Y_i = \sqrt{c_i}\na(h'_i(c_i)-h'_i(\bar{c}_i))$, $i = 1,\ldots,n$.

{\em Step 1: Preparations.} 
Recall that we have assumed that $\bar{c}_i  (x, t) \ge m$ for 
$x \in \Omega$, $t > 0$ for some $m > 0$.
Let $c_*(x,t):=\min_{i=1,\ldots,n}c_i(x,t)$. 
We split the estimations of the above integrals into two subdomains: one where 
$c_*(x,t)\le m/2$ and another one where $c_*(x,t)>m/2$.
To this end, we use a cutoff function. Let $\eps > 0$ be sufficiently small 
and $\psi : [0,1] \to [0,1]$ be a $C^2$-function,
which takes the values $\psi = 0$ on $[0,m/2]$,
$\psi = 1$ on $[m/2 + \eps, 1]$, and $\psi \in (0,1) $ on the complementary 
interval $(m/2,m/2 + \eps)$.
Define $\chi(\bm{c}) : \R^n \to [0,1]$ by
\begin{equation}\label{4.defcutoff}
  \chi(\bm{c}) := \prod_{i=1}^n\psi(c_i)
  = \begin{cases}
  1  &\quad\mbox{if } m/2 + \eps \le c_i \le 1\mbox{ for all } i = 1,\ldots, n, \\
  0  &\quad\mbox{if } 0 < c_i \le m/2 \mbox{ for some $i$,} \\
  \alpha(\bm{c}) \in (0,1)  & \quad\mbox{else}.
\end{cases} 
\end{equation}
We employ $\chi(\bm{c})$ to split the integral \eqref{4.I3I4I5} into two parts:
\begin{align}
  H(\bm{c}|\bar{\bm{c}})(t) - H(\bm{c}^0|\bar{\bm{c}}^0)
	&\le  \int_0^t\int_\Omega (1 - \chi (\bm{c}) ) (I_3+I_4 + I_5)dxds \label{4.split} \\
	&\phantom{xx}{}+ \int_0^t\int_\Omega \chi (\bm{c})  (I_3+I_4 + I_5) dxds 
	=: J_L + J_H. \nonumber
\end{align}
In the sequel, we estimate $J_L$ and $J_H$ separately. 

{\em Step 2: Case $c_*(x,t)\le m/2 + \eps$.}
We estimate the term $J_L$ in \eqref{4.split}.
By replacing $p(c_i|\bar{c}_i)$ in $I_5$ by definitions \eqref{4.defpress}
and \eqref{4.relfundef} and tracing backwards the derivation from
\eqref{4.relenineq2} to \eqref{4.relenineq}, we can express the
integral over $(1-\chi(\bm{c})(I_3+I_4+I_5)$ by $(1-\chi(\bm{c})(I_1+I_2)$ 
except for a term accounting for the cutoff function:
\begin{align}\label{4.JL}
  J_L &= \int_0^t \int_\Omega  (1 - \chi (\bm{c}) ) (I_1 + I_2 )dxds \\
  &\phantom{xx}{}+ \int_0^t \int_\Omega\na\chi(\bm{c})\cdot 
	\sum_{i=1}^n (p_i (c_i) - p_i (\bar{c}_i)) 
  \sum_{j=1}^n  B_{ij}^{BD}(\bar{\bm{c}})  \frac{\sqrt{\bar{c}_j}}{\sqrt{\bar{c}_i}}
	\na h'_j(\bar{c}_j). \nonumber
\end{align}

In the sequel, we estimate the right-hand side of \eqref{4.JL} term-by-term.
To estimate $I_1$, we set $\bm{Z}=(Z_1,\ldots,Z_n)$
with $Z_i:=\sqrt{c}_i\na h'_i(c_i)$, $i=1,\ldots,n$. By Lemma \ref{lem.AK},
the matrix $B^{BD}(\bm{c})$ is symmetric and positive semi-definite.
Therefore, using Young's inequality and the boundedness of $B^{BD}(\bm{c})$
(see Assumption (B2)),
\begin{align*}
  I_1 &= -\bm{Z}^T B^{BD}(\bm{c})\bm{Z} + \sum_{i,j=1}^n B_{ij}^{BD}(\bm{c})
	Z_i\sqrt{c_j}\na h'_j(\bar{c}_j) \\
	&\le -\frac12\bm{Z}^T B^{BD}(\bm{c})\bm{Z}
	+ \frac12\sum_{i,j=1}^n B_{ij}^{BD}(\bm{c})(\sqrt{c_i}\na h'_i(\bar{c}_i))\cdot
	(\sqrt{c_j}\na h'_j(\bar{c}_j)) \\
	&\le -\frac12\bm{Z}^T B^{BD}(\bm{c})\bm{Z} + C\sum_{i=1}^n|\na h'_i(\bar{c}_i)|^2,
\end{align*}
where $C>0$ depends on $m$ and $\mu$ (defined in Assumption (B4)). 
For the term $I_2$, we use the regularity for $\bar{c}_i$ to conclude that
$$
  I_2 \le \sum_{i=1}^n(c_i-\bar{c}_i)^2 + C\sum_{i,j=1}^n
	\big|h''_i(\bar{c}_i)\diver\big(\sqrt{\bar{c}_i}B_{ij}^{BD}(\bar{\bm{c}})
	\sqrt{\bar{c}_j}\na h'_j(\bar{c}_j)\big)\big|^2
	\le \sum_{i=1}^n(c_i-\bar{c}_i)^2 + C.
$$
On the set $\chi(\bm{c})<1$, we have $c_*(x,t)<m/2+\eps$, and there exists
$\ell\in\{1,\ldots,n\}$ such that $c_\ell(x,t)<m/2+\eps$. Thanks to Assumption
(H) on page \pageref{hyp.fnh}, we can apply Lemma \ref{lem.relenes} to find that
\begin{equation*}
  h_{i_0}(c_{i_0}|\bar{c}_{i_0}) \ge \kappa_m(c_{i_0}-\bar{c}_{i_0})^2 
	\ge \kappa_m \bigg(\frac{m}{2} - \eps\bigg)^2 \quad \mbox{when }\chi(\bm{c}) < 1.
\end{equation*}
We infer that
$$
  I_2 \le \sum_{i=1}^n(c_i-\bar{c}_i)^2 + C\sum_{i=1}^n h_i(c_i|\bar{c}_i)
  \le C\sum_{i=1}^n h_i(c_i|\bar{c}_i).
$$

It remains to estimate the last term in \eqref{4.JL}, using the fact that 
$\na\chi(\bm{c})$ vanishes outside the set $c_* \in [m/2,m/2+\eps]$,
the Lipschitz continuity of $p_i (c_i)$, entropy inequality \eqref{4.esdis2}
and Lemma \ref{lem.relenes}:
\begin{align*}
  &\bigg|\int_0^t \int_\Omega\na\chi(\bm{c})\cdot \sum_{i=1}^n 
	(p_i (c_i) - p_i (\bar{c}_i))\sum_{j=1}^n B_{ij}^{BD}(\bar{\bm{c}})  
	\frac{\sqrt{\bar{c}_j}}{\sqrt{\bar{c}_i}}\na h'_j(\bar{c}_j) dx ds\bigg| \\
  &\phantom{xx}{}\le C\bigg(\max_{j = 1,\ldots,n}\sup_{m/2\le c_j \le m/2 + \eps}
  \bigg| \frac{\pa\chi}{\pa c_j}(\bm{c})\bigg| \bigg)
  \int_0^t \int_\Omega \mathrm{1}_{\{m/2 < c_* < m/2 + \eps\}} 
	\sum_{i=1}^n |\na c_i|\sum_{j=1}^n |c_j - \bar{c}_j |dx ds \\
  &\phantom{xx}{}\le  C \int_0^t\int_\Omega\mathrm{1}_{\{c_* >  m/2\}}
	\sum_{i=1}^n |\na c_i|^2 dxds 
	+  C\int_0^t \int_\Omega \mathrm{1}_{\{ c_* < m/2 + \eps\}} 
  \sum_{j=1}^n |c_j - \bar{c}_j |^2 dx ds \\
  &\phantom{xx}{}\le C + \int_0^t \int_\Omega \mathrm{1}_{\{ c_* < m/2 + \eps\}} 
  \sum_{j=1}^n |c_j - \bar{c}_j |^2 dx ds \\
  &\phantom{xx}{}\le C \int_0^t \int_\Omega\mathrm{1}_{\{ c_* < m/2 + \eps\}}
	\sum_{i=1}^n h_i (c_i | \bar{c}_i) dx ds.
\end{align*}
Note that the final constant $C$, depending on 
$\max_{j = 1,\ldots,n} \sup_{m/2\le c_j \le m/2 + \eps}|(\pa\chi/\pa c_j)
(\bm{c})|$, will blow up if we let $\eps \to 0$. Therefore, we fix $\eps>0$.
Combining the previous estimate, we end up with
\begin{equation}\label{4.step2}
  J_L \le -\frac12\int_0^t \int_\Omega(1 - \chi(\bm{c}))\bm{Z}^T 
	B^{BD}(\bm{c})\bm{Z} dx ds 
	+ C \int_0^t \int_\Omega \mathrm{1}_{\{ c_* < m/2 + \eps\}}\sum_{i=1}^n 
	h_i(c_i|\bar{c}_i) dx ds.
\end{equation}

{\em Step 3: Case $c_*(x,t) > m/2$.}
We proceed to estimate the term $J_H$ in \eqref{4.split}. The range of integration 
now consists of the sets $\{m/2 < c_* < m/2 + \eps \}$, where 
$0 < \chi(\bm{c}) < 1$, and the set
$\{m/2 + \eps < c_* \le 1 \}$, where $\chi(\bm{c}) =1$.

For the term $I_3$, we use \eqref{4.esdis3} in Lemma \ref{lem.dis}:
\begin{align*}
  -\int_0^T&\int_\Omega\chi(\bm{c})I_3 dxds
	= \int_0^t\int_\Omega\chi(\bm{c})\sum_{i,j=1}^n B^{BD}_{ij}(\bm{c}) 
	Y_i\cdot Y_j dx ds \\
  &\ge \beta(m)\int_0^t\int_\Omega\chi(\bm{c})\sum_{i=1}^n|\na(c_i - \bar{c}_i)|^2 
	dx ds - C \int_0^t\int_\Omega\chi(\bm{c})\sum_{i=1}^n |c_i - \bar{c}_i|^2 dx ds.
\end{align*}
By Young's inequality with $\delta>0$, the term $I_4$ can be estimated as
\begin{align*}
  \int_0^t&\int_\Omega\chi(\bm{c})I_4 dxds
  \le \delta\sum_{i=1}^n\int_0^t\int_\Omega\chi(\bm{c})|Y_i|^2 dxds \\
  &{} + \frac{1}{4\delta}\sum_{i,j=1}^n\int_0^t\int_\Omega
	\chi(\bm{c})\bigg(B_{ij}^{BD}(\bm{c})\sqrt{c_j}
	- \frac{\sqrt{c_i}B_{ij}^{BD}(\bar{\bm{c}})\sqrt{\bar{c}_j}}{\sqrt{\bar{c}_i}}
	\bigg)^2|\na h'_j(\bar{c}_j)|^2 dxds.
\end{align*}
Recall that we work in the range $c_i > m/2$ and $\bar{c}_i \ge m$. 
The boundedness and Lipschitz continuity of $h_i''$ imply that
\begin{align*}
  \sum_{i=1}^n |Y_i|^2 &= \sum_{i=1}^n c_i|\na(h_i'(c_i)-h'_i(\bar{c}_i))|^2 
	\le \sum_{i=1}^n c_i|h''_i(c_i)\na(c_i-\bar{c}_i)
	+ (h''_i(c_i)-h''_i(\bar{c}_i))\na\bar{c}_i|^2 \\
	&\le C\sum_{i=1}^n\big(|\na(c_i-\bar{c}_i)|^2 + (c_i-\bar{c}_i)^2|\na\bar{c}_i|^2
	\big).
\end{align*}
Furthermore, the boundedness and Lipschitz continuity of $B_{ij}^{BD}$ 
(see Lemma \ref{lem.AK}) yield
\begin{align*}
  \bigg|B_{ij}^{BD}&(\bm{c})\sqrt{c_j}	- \frac{\sqrt{c_i}B_{ij}^{BD}(\bar{\bm{c}})
	\sqrt{\bar{c}_j}}{\sqrt{\bar{c}_i}}\bigg| \\
  &= \bigg|B_{ij}(\bm{c})\sqrt{c_j}	- B_{ij}^{BD}(\bar{\bm{c}})\sqrt{\bar{c}_j} 
	+	\frac{(\sqrt{\bar{c}_i}-\sqrt{{c}_i})B_{ij}^{BD}(\bar{\bm{c}})
	\sqrt{\bar{c}_j}}{\sqrt{\bar{c}_i}}\bigg| \\
  &\le C \sum_{i=1}^n |c_i - \bar{c}_i| + |\sqrt{c_i} - \sqrt{ \bar{c}_i} |
  \le C \sum_{i=1}^n |c_i - \bar{c}_i|.
\end{align*}
Thus, the choice $\delta = \beta(m)/(2C)$ gives
\begin{align*}
  \int_0^t\int_\Omega \chi(\bm{c})I_4 dxds
  &\le \frac{\beta(m)}{2}\int_0^t\int_\Omega\chi(\bm{c})\sum_{i=1}^n  
	|\na(c_i-\bar{c}_i)|^2 dxds \\
  &\phantom{xx}{} + C \int_0^t\int_\Omega
	\chi(\bm{c})\sum_{i=1}^n |c_i - \bar{c}_i|^2 dxds.
\end{align*}
Finally, we use definition \eqref{4.relfundef} of $p_i(c_i|\bar{c}_i)$
and Hypothesis \eqref{hyp.fnh} to estimate
\begin{align*}
  |p_i(c_i|\bar{c}_i)| &= \bigg|(c_i - \bar{c}_i)^2 \int_0^1 \int_0^s 
	p_i''(\tau c_i + (1-\tau)\bar{c}_i) d\tau ds \bigg| \\
  &\le K_2 (c_i - \bar{c}_i)^2 \int_0^1 \int_0^s h_i''(\tau c_i + (1-\tau) 
	\bar{c}_i) d\tau ds =  K_2 h_i (c_i|\bar{c}_i).
\end{align*}
In turn, this implies that
$$
  \int_0^t\int_\Omega \chi (\bm{c})  I_5  dxds
  \le C \int_0^t\int_\Omega\chi(\bm{c})\sum_{i=1}^n h_i(c_i|\bar{c}_i) dxds.
$$
Summarizing the previous computations and using Lemma \ref{lem.relenes}, 
we conclude that
\begin{equation} \label{4.step3}
  J_H \le - \frac{\beta(m)}{2}\int_0^t\int_\Omega\chi(\bm{c}) \sum_{i=1}^n 
	|\na(c_i - \bar{c}_i)|^2 dx ds
  + C \int_0^t\int_\Omega\chi(\bm{c})\sum_{i=1}^n h_i(c_i|\bar{c}_i) dxds.
\end{equation}

{\em Step 4: End of the proof.} We combine the differential inequality 
\eqref{4.split} with the estimations \eqref{4.step2} and \eqref{4.step3} to obtain
\begin{align}\label{4.improved}
	H(\bm{c}&(t)|\bar{\bm{c}}(t)) 
	  + \frac12\int_0^t\int_\Omega(1-\chi(\bm{c}))
	  \bm{Z}^T B^{BD}(\bm{c})\bm{Z} dxds \\
	  &\phantom{xxxxxx}{}+ \frac{\beta(m)}{2}
	  \int_0^t\int_\Omega \chi(\bm{c})
		\sum_{i=1}^n |\na(c_i-\bar{c}_i)|^2dxds \nonumber \\
	  &\le H(\bm{c}^0|\bar{\bm{c}}^0) + C\int_0^t H(\bm{c}|\bar{\bm{c}})ds, \nonumber
  \end{align}  
The constant $C>0$ depends in particular on $m$ and
the $L^\infty(0,T;W^{2,\infty}(\Omega))$ norm of $\bar{c}_j$, $j=1,\ldots,n$.
The proof of Theorem \ref{thm.gws} finishes after applying Gronwall's inequality.

\begin{remark}\rm
Inequality \eqref{4.improved} leads to a slightly stronger version of the
relative entropy inequality than stated in Theorem \ref{thm.gws}. However, we obtain
gradient estimates only on the set $\{c_*>m/2\}$, while on $\{c_*\le m/2\}$,
the quadratic form $\bm{Z}^TB^{BD}(\bm{c})\bm{Z}$ generally does not
lead to a control of the $L^2$-norm of $\na c_i$. 
\qed
\end{remark}


\section{Examples}\label{sec.exam}

We present some examples for the generalized Maxwell--Stefan system
\eqref{1.gms1}--\eqref{1.gms2} satisfying Assumptions (B1)--(B4).

\subsection{A cross-diffusion system for thin-film solar cells}

Thin-film crystalline solar cells can be fabricated by the so-called physical
vapor deposition process. This process produces a metal vapor that can be deposited
on electrically conductive materials as a thin coating. It is shown in
\cite{BaEh18} that the evolution of the volume fractions of the thin-film components
can be described by the cross-diffusion system
\begin{equation}\label{pvd1}
  \pa_t c_i = \diver\bigg(\sum_{j=1}^n a_{ij}(u_j\na u_i-u_i\na u_j)\bigg), 
	\quad i=1,\ldots,n,
\end{equation}
where $a_{ij}=a_{ji}>0$ for $i,j=1,\ldots,n$, and $\sum_{i=1}^n c_i=1$. 
This model can be formulated 
as a generalized Maxwell--Stefan system. Indeed, let $h_i(c_i)=c_i(\log c_i-1)$
and $K_{ij}(\bm{c}) = \sum_{j=1}^n\sqrt{c_i}A_{ij}^{BD}(\bm{c})\sqrt{c_j}$ 
for $i,j=1,\ldots,n$,
where $A(\bm{c})$ is given by \eqref{1.A} with $D_{ij}=1/a_{ij}$.
Then $B(\bm{c})=A^{BD}(\bm{c})$ (see \eqref{1.AK}) and hence
$B^{BD}(\bm{c})=A(\bm{c})$. Equation \eqref{4.eqABD} becomes
\begin{equation}\label{pvd2}
  \pa_t c_i = \diver\bigg(\sum_{j=1}^n \sqrt{c_i}A_{ij}(\bm{c})\sqrt{c_j}
	\na\log c_j\bigg), \quad i=1,\ldots,n.
\end{equation}
Because of \eqref{1.A}, the mobility matrix $(\sqrt{c_i}A_{ij}(\bm{c})\sqrt{c_j})$
reads as
$$
  \sqrt{c_i}A_{ij}(\bm{c})\sqrt{c_j} = \left\{\begin{array}{ll}
	\sum_{k\neq i} a_{ik}c_ic_k &\quad\mbox{if }i=j, \\
	-a_{ij}c_ic_j &\quad\mbox{if }i\neq j,
	\end{array}\right.
$$
and an elementary computation shows that \eqref{pvd2} can actually be written as
\eqref{pvd1}. 

Although it can be checked that the matrix
$B(\bm{c})=A^{BD}(\bm{c})$ satisfies Assumptions (B1)--(B4),
we can here directly verify the statements of Lemma \ref{lem.AK}.
Definition \eqref{1.A} of $A(\bm{c})$ immediately implies that
$B_{ij}^{BD}(\bm{c})$ is bounded and Lipschitz continuous on $[0,1]^n$.
Property \eqref{4.ABDpd} follows from \eqref{2.A} in Lemma \ref{lem.A} with
$\lambda(m)=\mu$. Hence, the weak-strong uniqueness property holds for this
model. 

\subsection{A tumor-growth model}
 
The growth of a symmetric avascular tumor can be modeled on the mechanical level
by diffusion fluxes of the tumor cells, the extracellular matrix (ECM), and
the interstitial fluid (water, nutrients, etc.). The model was suggested in
\cite{JaBy02} and analyzed in \cite{JuSt13}. The evolution of the volume fractions 
$c_i$ of the tumor cells, ECM, and interstitial fluid is given by
(see, e.g., \cite[Section 4.2]{Jue16})
\begin{align}
  \pa_t c_i + \diver(c_iu_i) &= 0, \quad i=1,\ldots,3, \label{tumor1} \\
	\na(c_iP_i) + c_i\na p &= -\sum_{j=1}^3 k_{ij}c_ic_j(u_i-u_j), \quad i=1,2, 
	\label{tumor2} \\
	c_3\na p &= -\sum_{j=1}^3 k_{ij}c_ic_j(u_i-u_j), \label{tumor3}
\end{align}
where $k_{ij}=k_{ji}>0$ for $i,j=1,2,3$, 
the partial pressures $P_1$, $P_2$ and the phase pressure $p$ are given by
$$
  P_1 = c_1, \quad P_2 = \beta c_2(1+\theta c_1), \quad
	p = -c_1P_1-c_2P_2,
$$
$\beta>0$, $\theta>0$ are suitable parameters, and it holds that $\sum_{i=1}^3c_i=1$.

We claim that \eqref{tumor1}--\eqref{tumor3} can be formulated
as a generalized Maxwell--Stefan system. We define the entropy densities
as in the previous example, $h_i(c_i)=c_i(\log c_i-1)$, $i=1,2,3$. 
With the matrix
$$
  W(\bm{c}) = \begin{pmatrix}
	2c_1(1-c_1)-\beta\theta c_1c_2^2 & -2\beta c_1c_2(1+\theta c_1) & 0 \\
  -2c_1c_2+\beta \theta(1-c_2)c_2^2 & 2\beta c_2(1-c_2)(1+\theta c_1) & 0 \\
  -2c_1c_3-\beta\theta c_3c_2^2 & -2\beta c_3c_2(1+\theta c_1) & 0
	\end{pmatrix},
$$
the left-hand side of \eqref{tumor2}--\eqref{tumor3} can be written in a 
more concise form:
$$
  \begin{pmatrix}
	\na(c_1P_1) + c_1\na p \\ \na(c_2P_2) + c_2\na p \\ c_3\na p
	\end{pmatrix}
	= W(\bm{c})\begin{pmatrix} \na c_1 \\ \na c_2 \\ \na c_3 \end{pmatrix}.
$$
Let the matrix $A(\bm{c})$ be given by \eqref{1.A} with $D_{ij}=1/k_{ij}$. 
Then the right-hand side of
\eqref{tumor2}--\eqref{tumor3} equals (also see \eqref{1.eqA})
$$
  -\sum_{j=1,\,j\neq i}^3 \frac{c_ic_j}{D_{ij}}(u_i-u_j)
	= -\sum_{j=1}^3\sqrt{c_i}A_{ij}(\bm{c})\sqrt{c_j}u_j.
$$
Thus, inverting 
$$
  \sum_{j=1}^n W_{ij}(\bm{c})\na c_j 
	= -\sum_{j=1}^3\sqrt{c_i}A_{ij}(\bm{c})\sqrt{c_j}u_j, \quad i=1,2,3,
$$
(which is the same as \eqref{tumor2}--\eqref{tumor3}) yields
$$
  \sqrt{c_i}u_i = -\sum_{j,k=1}^3 A_{ij}^{BD}(\bm{c})\frac{1}{\sqrt{c_j}}
	W_{jk}(\bm{c})\na c_k, \quad i=1,2,3,
$$
and system \eqref{tumor1}--\eqref{tumor3} can be written for $i=1,2,3$ as
\begin{align*}
  \pa_t c_i &= \diver\bigg(\sum_{j,k=1}^3\sqrt{c_i}A_{ij}^{BD}(\bm{c})
	\frac{1}{\sqrt{c_j}}W_{jk}(\bm{c})\na c_k\bigg) \\
	&= \diver\bigg(\sum_{j,k,\ell=1}^3\sqrt{c_i}A_{ij}^{BD}(\bm{c})
	\frac{1}{\sqrt{c_j}}W_{jk}(\bm{c})\sqrt{c_k}(P_L)_{k\ell}\sqrt{c_\ell}
	\na\log c_\ell\bigg),
\end{align*}
recalling definition \eqref{2.PL} of $P_L$. Thus,
\begin{align*}
  & \pa_t c_i = \diver\bigg(\sum_{\ell=1}^n \sqrt{c_i}R_{i\ell}(\bm{c})\sqrt{c_\ell}
	\na\log c_\ell\bigg), \quad\mbox{where} \\
	& R_{i\ell}(\bm{c}) := \sum_{j,k=1}^3 A_{ij}^{BD}(\bm{c})
	\frac{1}{\sqrt{c_j}}W_{jk}(\bm{c})\sqrt{c_k}(P_L)_{k\ell}, \quad i,\ell=1,2,3.
\end{align*}
Also in this case, it is more convenient to check the statements of Lemma \ref{lem.AK}
instead of Assumptions (B1)--(B4). Notice that $W(\bm{c})$ is not symmetric,
so $R(\bm{c})$ is not symmetric either. The elements $R_{ij}(\bm{c})$
are bounded and Lipschitz continuous, since $A_{ij}^{BD}(\bm{c})$ and
$W_{ij}(\bm{c})$ have these properties. The second statement of Lemma \ref{lem.AK},
namely property \eqref{4.ABDpd}, is verified only for a special example that
was considered in \cite{JuSt13}.

\begin{lemma}
Let $k_{ij}=1$ for $i,j=1,2,3$ and $0\le\theta<4/\sqrt{\beta}$. Then,
with $m=\min_{i=1,2,3}c_i>0$, there exists $\lambda(m)>0$ such that
$$
  \bm{z}^T R(\bm{c})\bm{z} \ge \lambda(m)|P_L\bm{z}|^2 \quad\mbox{for all }
	\bm{z}\in\R^3.
$$
\end{lemma}

\begin{proof}
The assumption $k_{ij}=1$ implies that $A(\bm{c})=P_L$ and hence
$A^{BD}(\bm{c})=P_L(A(\bm{c})P_L+P_{L^\perp})^{-1}=P_L$. Suppose that
for any $\bm{y}\in\R^3$ satisfying $\bm{y}\in L$ (i.e.\ $\sqrt{\bm{c}}\cdot\bm{y}=0$),
we have
\begin{equation}\label{4.Qpd}
  \sum_{i,j=1}^3\frac{1}{\sqrt{c_i}}W_{ij}(\bm{c})\sqrt{c_j}y_iy_j
	\ge \lambda(m)|\bm{y}|^2.
\end{equation}
Then, for $\bm{z}\in\R^3$ and $\bm{y}=P_L\bm{z}\in L$,
$$
  \bm{z}^T R(\bm{c})\bm{z}
	= \sum_{i,j=1}^3(P_L\bm{z})_i\frac{1}{\sqrt{c_i}}W_{ij}(\bm{c})\sqrt{c_j}
	(P_L\bm{z})_j \ge \lambda(m)|P_L\bm{z}|^2,
$$
which proves the lemma. 

It remains to verify \eqref{4.Qpd}. Using 
$y_3=-(\sqrt{c_1}y_1+\sqrt{c_2}y_2)/\sqrt{c_3}$, we calculate
$$
  \sum_{i,j=1}^3\frac{1}{\sqrt{c_i}}W_{ij}(\bm{c})\sqrt{c_j}y_iy_j
	= 2\beta(1+\theta c_1)(\sqrt{c_2} y_2)^2 + \beta\theta c_2(\sqrt{c_1}y_1)
	(\sqrt{c_2}y_2) + 2(\sqrt{c_1}y_1)^2.
$$
Since $0\le\theta<4/\sqrt{\beta}$, the discriminant of the quadratic form
is negative and there exists $\kappa>0$ such that
\begin{align*}
  \sum_{i,j=1}^3\frac{1}{\sqrt{c_i}}W_{ij}(\bm{c})\sqrt{c_j}y_iy_j
	&\ge \kappa(c_1y_1^2 + c_2 y_2^2) 
	\ge \frac{\kappa}{3}\big(c_1y_1^2 + c_2y_2^2 
	+ (\sqrt{c_1}y_1+\sqrt{c_2}y_2)^2\big) \\
	&= \frac{\kappa}{3}(c_1y_1^2+c_2y_2^2+c_3y_3^2)
	\ge \frac{\kappa}{3}\Big(\min_{i=1,2,3}c_i\Big)(y_1^2+y_2^2+y_3^2),
\end{align*}
proving the claim with $\lambda(m)=\kappa m/3$.
\end{proof}

We deduce from the previous lemma that the weak-strong uniqueness property
holds for the tumor-growth model if $k_{ij}=1$ for $i,j=1,2,3$ and 
$0\le\theta<4/\sqrt{\beta}$. The latter condition is necessary to achieve
the global existence of weak solutions, since it guarantees the positive
semidefiniteness of the mobility matrix; see \cite{JuSt13} for details.

\subsection{A multi-species porous-medium-type model}

Another model is a generalization of the first example to illustrate that
also non-logarithmic entropies may be considered. We choose
$h_i(c_i)=c_i^{\gamma}/(\gamma-1)$ with $\gamma>1$ and $A(\bm{c})$ as
in \eqref{1.A}. The partial pressure
becomes $p_i=c_ih'_i(c_i)-h_i(c_i)=c_i^\gamma$, and equation \eqref{pvd2}
reads here as
$$
  \pa_t c_i = \diver\bigg(\sum_{j=1}^n\sqrt{c_i}A_{ij}(\bm{c})\sqrt{c_j}
	\na h_j'(c_j)\bigg)
	= \frac{\gamma}{\gamma-1}\diver\bigg(\sum_{j=1}^n\frac{1}{D_{ij}}
	(c_j\na c_i^{\gamma-1} - c_i\na c_j^{\gamma-1})\bigg).
$$
Hence, the weak-strong property holds for this model.

\subsection{Maxwell-Stefan system with different molar masses} 

In equations \eqref{1.eq}, we have implicitly assumed that all molar masses
of the species are the same. We show that the weak-strong uniqueness property also 
holds for the model proposed in \cite{BCH54,BoDr15} without this assumption. In the 
case of different molar masses $M_i$, we need to distinguish between the mass 
densities $\rho_i$ and the molar concentrations $c_i=\rho_i/M_i$. 
The Maxwell--Stefan equations read for $i=1,\ldots,n$ as 
\begin{equation}\label{5.4}
  \pa_t\rho_i + \diver(\rho_i u_i) = 0, \quad
  -\sum_{j=1}^n \frac{c_ic_j}{c^2 D_{ij}}(u_i-u_j) 
	= \rho_i\na\frac{\delta H}{\delta \rho_i}(\bm{\rho}) 
	- \rho_i \sum_{j=1}^n \rho_j\na\frac{\delta H}{\delta \rho_j}(\bm{\rho}),
\end{equation}
where $c=\sum_{i=1}^n c_i$. 
As before, the restriction $\sum_{j= 1}^n \rho_j = 1$ inherited from the initial data
is imposed. The second equation can be rewritten as 
$$
	-\sum_{j=1}^n \frac{\rho_i\rho_j}{\widetilde{D}_{ij}(\bm{\rho})} (u_i-u_j) 
	= \rho_i\na\frac{\delta H}{\delta \rho_i}(\bm{\rho}) 
	- \rho_i \sum_{j=1}^n \rho_j\na\frac{\delta H}{\delta \rho_j}(\bm{\rho}),
$$
where $\widetilde{D}_{ij}(\bm{\rho}) = c^2M_iM_j = (\sum_{k=1}^n \rho_k/M_k)^2M_iM_j$. 
Due to
$$
  \sum_{k=1}^n \frac{\rho_k}{M_k} \ge 
	\sum_{k=1}^n \frac{\rho_k}{\max_{\ell=1,\ldots,n}{M_\ell}} 
	= \frac{1}{\max_{\ell=1,\ldots,n}M_\ell},
$$ 
the coefficients $\widetilde{D}_{ij}(\bm{\rho})$ are uniformly bounded from below.
Since the proof of Lemma \ref{lem.A} only relies on the uniform boundness of $D_{ij}$,
Lemma \ref{lem.A} also  holds for the following matrix $\widetilde{A}(\bm{c})$, 
defined similarly as in \eqref{1.A}:
\begin{equation*}
	\widetilde{A}_{ij}(\bm{\rho}) = \left\{\begin{array}{ll}
	  \sum_{k=1,\,k\neq i}^n \rho_k/\widetilde{D}_{ik}(\bm{\rho}) &\quad\mbox{if }i=j, \\
	  -\sqrt{\rho_i\rho_j}/\widetilde{D}_{ij}(\bm{\rho}) &\quad\mbox{if }i\neq j.
	  \end{array}\right.
\end{equation*}
Therefore, the weak-strong uniqueness holds if $H(\bm{c})$ satisfies the assumptions of 
Theorem \ref{thm.gws}.

Recalling that $H(\bm{\rho}) = \sum_{i=1}^n \int_\Omega h_i(\rho_i) dx$, we may
formulate the entropy in terms of the concentrations $\bm{c}$ as 
$\eta(\bm{c}) = \sum_{i=1}^n \int_\Omega \eta_i(c_i) dx,$ where 
$\eta_i(c_i) = h_i(\rho_i/M_i)$. Then we can rewrite the second equation in 
\eqref{5.4} as 
$$
	  -\sum_{j=1}^n \frac{c_ic_j}{c^2D_{ij}}(u_i-u_j) 
		= c_i \na\mu_i - c_i M_i \sum_{j=1}^n c_j \na\mu_j,
$$
where $\mu_i = \eta_i'(c_i)$ is the molar-based chemical potential. 
Using the Gibbs--Duhem equation $\na p = \sum_{i=1}^n c_j\na\mu_j$, where $p$ 
is the pressure, the above equation can be put into the form 
$$
	  -\sum_{j=1}^n \frac{c_ic_j}{c^2 D_{ij}}(u_i-u_j) 
		= c_i\nabla \mu_i - \rho_i \nabla p,
$$
which is \cite[Formula (203)]{BoDr15}. Yet another formulation in terms of the
molar fractions $X_i=c_i/c$ is
$$
	-\sum_{j=1}^n \frac{c_ic_j}{c^2 D_{ij}}(u_i-u_j) 
	= c_i\nabla_p\widetilde{\mu}_i + (\phi_i - \rho_i)\nabla p,
$$
where $\widetilde{\mu}_i$ is given by $\widetilde{\mu}_i(p,X_1,\ldots,X_n) 
= \mu_i(c_i)$, $\nabla_p \widetilde{\mu}_i := \sum_{j=1}^n(\pa\widetilde{\mu}_i/
\partial X_j)\nabla X_j$, and $\phi_i:=\pa\widetilde{\mu}_i/\pa p$ is the 
volume fraction.

A simple choice is the entropy 
$$
  \eta_i(c_i) = c_i\log c_i - c_i, \quad i=1,\ldots,n,
$$
corresponding to $h_i(\rho_i) = (\rho_i/M_i)(\log(\rho_i/M_i) - 1)$.
It leads to $\mu_i=\log c_i$, $p=c$, and the model 
$$
	-\sum_{j=1}^n \frac{c_ic_j}{c^2 D_{ij}}(u_i-u_j) = \nabla c_i - \rho_i \nabla c.
$$
The existence of local strong solutions to this model can be proved as in
\cite{Bot11}, while the existence of global weak solutions was shown in
\cite{ChJu15}.


\begin{appendix}
\section{The Bott--Duffin inverse}\label{app}

For the convenience of the reader, we recall the definition and some properties
of the Bott--Duffin inverse. Let $A\in\R^{n\times n}$ be an arbitrary matrix and
$L\subset\R^n$ be a subspace. 
The Bott--Duffin inverse is introduced in connection
to the solution of the constrained inversion problem 
(see \cite{BD53}, \cite[Ch 2.10]{BIG74})
\begin{equation}\label{A.1}
  A x + y = b, \quad x \in L, \quad y \in L^\perp.
\end{equation}
Let $P_L$ and $P_{L^\perp}$ be
the projection operators onto $L$ and $L^\perp$, respectively.
The set of solutions of \eqref{A.1} is the same as the set of solutions to
$(AP_L + P_{L^\perp})z = b$,
and $(x,y)$ solves \eqref{A.1} if and only if 
$x = P_L z$ and $y = P_{L^\perp} z = b - AP_L z$.
Then, if the matrix $AP_L+P_{L^\perp}$ is invertible, 
we define the {\em Bott--Duffin inverse} of $A$ with respect to $L$ by
\begin{equation}\label{A.BD}
  A^{BD} := P_L(AP_L+P_{L^\perp})^{-1}.
\end{equation}
and the solution to \eqref{A.1} is expressed in the form
\begin{equation}\label{A.3}
x = A^{BD} b \, , \quad y = b - Ax.
\end{equation}
If $L=\ran(A)$ and $A$ is symmetric, the Bott--Duffin inverse 
is the same as the group inverse, which was investigated in the context of
Maxwell--Stefan systems in \cite{BoDr20}.

Let $A$ be symmetric. We call $A$ {\em $L$-positive definite} if
$\bm{z}^TA\bm{z}>0$ for all $\bm{z}\in L\setminus\{\bm{0}\}$. 
For this class of matrices, a generalized Bott--Duffin inverse is defined in 
\cite{Yon90}, which coincides with the classical Bott--Duffin inverse when 
$AP_L+P_{L^\perp}$ is invertible. The following result is proved in 
\cite[Lemma 2c and 1b]{Yon90}.

\begin{lemma}\label{lem.bott1}
Let $A$ be symmetric and $L$-positive definite. Then 
\begin{align*}
  {\rm (i)} &\ A^{BD}P_{L^\perp} = 0, \\
	{\rm (ii)} &\ \ran(AP_L+P_{L^\perp}) = P_L\ran(A)\otimes L^\perp, \quad
	\ker(AP_L+P_{L^\perp}) = \ker(AP_L)\cap L.
\end{align*}
\end{lemma}

It follows from property (i) that $A^{BD}$ can be formulated as
\begin{equation}\label{app.bd}
  A^{BD} = P_L(AP_L+P_{L^\perp})^{-1}(P_L+P_{L^\perp}) 
	= P_L(AP_L+P_{L^\perp})^{-1}P_L = A^{BD}P_L.
\end{equation}

\begin{lemma}\label{lem.bott2}
Let $A$ be symmetric and $L=\ran A$, $L^\perp=\ker A$. 
Then $AP_L=A$, $P_LA=A$, $A^{BD}$ is well defined and symmetric.
\end{lemma}

\begin{proof}
The identities $AP_L=A$ and $P_LA=A$ follow immediately from $L=\ran A$,
We infer from property (ii) that
$$
	\ker(AP_L+P_{L^\perp}) = \ker(AP_L)\cap L = \ker(A)\cap L = L^\perp\cap L
		= \{0\},
$$
showing that $AP_L+P_{L^\perp}$ is invertible. The matrix $AP_L = P_LAP_L$ is symmetric,
since $P_L$ and $A$ are symmetric. Also $P_{L^\perp}$ is symmetric,
so $AP_L+P_{L^\perp}$ and its inverse are symmetric too. Taking into account
\eqref{app.bd}, this implies that
$A^{BD} = P_L(AP_L+P_{L^\perp})^{-1}P_L$ is also symmetric.
\end{proof}

In our context, we are interested in the constrained inversion 
$$
  A x = b,  \quad x \in L,
$$
where $A$ is a symmetric positive semidefinite matrix, 
with $L = \ran(A)$ and thus $L^\perp = \ker(A)$, and $b \in L$.
Lemma \ref{lem.bott2} implies that $AP_L + P_{L^\perp}$ is invertible and 
$A^{BD}$ is well defined by \eqref{A.BD}. 
Because of \eqref{A.3}, we can express the inverse as
$x=A^{BD}b$ if $b\in L$.


\section{Pointwise estimates for entropy functions}\label{app.ent}

For the convenience of the reader, we recall the following lower bounds.

\begin{lemma}\label{lem.ent}
The following estimates hold for any $c$, $\bar{c}\in[0,1]$:
$$
  c\log\frac{c}{\bar{c}} - (c-\bar{c}) \ge \frac12(c-\bar{c})^2, \quad
	c\log\frac{c}{\bar{c}} - (c-\bar{c}) \ge (\sqrt{c}-\sqrt{\bar{c}})^2.
$$
\end{lemma}

\begin{proof}
Let $f(c)=c\log c$. Then
$$
  f(c)-f(\bar{c}) = f(\theta(c-\bar{c})+\bar{c})\big|_{\theta=0}^1
	= (c-\bar{c})\int_0^1 f'(\theta(c-\bar{c})+\bar{c})d\theta
$$
and 
\begin{align*}
  c\log&\frac{c}{\bar{c}} - (c-\bar{c}) = f(c)-f(\bar{c}) - f'(\bar{c})(c-\bar{c}) \\
	&= (c-\bar{c})\int_0^1\big(f'(\theta(c-\bar{c})+\bar{c}) - f'(\bar{c})\big)d\theta \\
	&= (c-\bar{c})\int_0^1 f'(s(c-\bar{c})+\bar{c})\big|_{s=0}^\theta d\theta 
	= (c-\bar{c})^2\int_0^1\int_0^\theta f''(s(c-\bar{c})+\bar{c})dsd\theta.
\end{align*}
The first inequality follows after observing that $f''(s(c-\bar{c})+\bar{c})
= 1/(s(c-\bar{c})+\bar{c})\ge 1$.

For the second inequality, we define $g(c)=(c\log c-c+1)/(\sqrt{c}-1)^2$ for
$c\neq 1$ and $g(1)=2$. Then $g$ is continuous and increasing, which implies that
$g(c)\ge g(0)=1$ and proves the statement.
\end{proof}

\begin{lemma} \label{lem.relenes}
Let $\bm{c},\bar{\bm{c}}\in \mathbb{R}_+^n$ satisfy $0\le c_i\le 1$, 
$m\le \bar{c}_i \le 1$, for $i = 1,\ldots,n$,  
and suppose that $h_i \in C([0,1]) \cap C^2((0,1])$ satisfies 
$$
  h_i''(c_i) > 0 \quad\mbox{for }0 < c_i \le 1.
$$
Then, for some $\kappa_m > 0$,
\begin{equation}\label{4.relbound}
	h_i(c_i|\bar{c}_i) = h_i(c_i) - h_i(\bar{c}_i) - h_i'(\bar{c}_i)(c_i-\bar{c}_i)  
	\ge \kappa_m (c_i-\bar{c}_i)^2.
\end{equation}
\end{lemma}
	
\begin{proof}
By Taylor expansion, the relative entropy density satisfies
$$
  \lim_{c_i\to\bar{c}_i}\frac{h_i(c_i|\bar{c}_i)}{(c_i-\bar{c}_i)^2}
	= \lim_{c_i\to\bar{c}_i}\int_0^1\int_0^\theta h_i''(s(c_i-\bar{c}_i)+\bar{c}_i)
	dsd\theta = \frac12 h_i''(\bar{c}_i) > 0.
$$
Therefore, $h_i(c_i|\bar{c}_i)/(c_i-\bar{c}_i)^2$ is a continuous function with
a positive minimum:
$$
  \kappa_m := \min_{i=1,\ldots,n}\min_{c_i\in[0,1],\,\bar{c}_i\in[m,1]}
	\frac{h_i(c_i|\bar{c}_i)}{(c_i-\bar{c}_i)^2}  > 0.
$$
This shows that $h_i(c_i|\bar{c}_i)\ge \kappa_m(c_i-\bar{c}_i)^2$ for 
$c_i \in [0,1]$, $\bar{c}_i \in [m,1]$ and proves \eqref{4.relbound}.
\end{proof}


\section{Thermodynamic derivation of the generalized Maxwell--Stefan system}
\label{app.thermo}

The aim of this section is to derive \eqref{1.gms1}--\eqref{1.gms2} from
elementary thermodynamic principles. We assume that the evolution of the gaseous
mixture is given by the conservation of mass and energy (without chemical reactions),
\begin{align}
  \pa_t(\rho c_i) + \diver(\rho c_i v+ J_i) &= 0, \nonumber \\
	\pa_t(\rho U) + \diver(\rho Uv + q) &= 0, \label{app.cons} \\
	\pa_t\rho + \diver(\rho v) &= 0, \quad i=1,\ldots,n, \nonumber
\end{align}
where $\rho_i$ is the partial density of the $i$th species, $\rho=\sum_{i=1}^n\rho_i$
the total density, $c_i=\rho_i/\rho$ the concentration of the $i$th species,
$v$ the barycentric velocity, $J_i$ the $i$th flux, $q$ the
heat flux, and the internal energy $U$ is given by the first law of thermodynamics 
in differential form by
\begin{equation}\label{app.U}
  dU = TdS - pdV + \sum_{i=1}^n \mu_i dc_i,
\end{equation}
where $S$ is the entropy, $V=1/\rho$ the volume, and $\mu_i=\pa U/\pa c_i$ 
the $i$th chemical potential. By definition, it holds that $\sum_{i=1}^n c_i=1$.
Adding the first and last equation in \eqref{app.cons}, we see that
$\diver\sum_{i=1}^n J_i=0$, which motivates us to assume that $\sum_{i=1}^nJ_i=0$.

The sum of the fluxes should vanish, $\sum_{i=1}^n J_i=0$, to be consistent
with the conservation laws.

With the material derivative $D_tf=\pa_t f+ v\cdot\na f$, the conservation laws
can be simplified to
$$
  \rho D_tc_i + \diver J_i = 0, \quad \rho D_t U + \diver q = 0, \quad
	D_t\rho + \rho\diver v = 0.
$$
Inserting these equations into equation \eqref{app.U}, formulated as
$D_t U = TD_tS - pD_tV+\sum_{i=1}^n\mu_iD_t c_i$, yields the entropy balance
\begin{align*}
  \rho D_t S &= \frac{\rho}{T}D_t U + \frac{\rho}{T}pD_t\bigg(\frac{1}{\rho}\bigg)
	- \sum_{i=1}^n\frac{\mu_i}{T}D_t c_i \\
  &= -\frac{1}{T}\diver q + \frac{p}{T}\diver v 
	+ \sum_{i=1}^n\frac{\mu_i}{T}\diver J_i
	= -\diver J_S + r_S, 
\end{align*}
where
$$
  J_S = \frac{q}{T} - \sum_{i=1}^n\frac{\mu_i}{T}J_i, \quad
	r_S = q\cdot\na\frac{1}{T} + p\diver v + \sum_{i=1}^n J_i\cdot \na\frac{\mu_i}{T}
$$
are the entropy flux and entropy production, respectively.

In our Maxwell--Stefan model, we assume that $v=0$ and $T=1$. Then the
entropy production simplifies to $r_S=\sum_{i=1}^n J_i\cdot \na\mu_i$.
It can be reformulated by taking into account that $\sum_{i=1}^n J_i=0$
and hence $J_i/\sqrt{c_i}\in L=\{\bm{x}\in\R^n:\sqrt{\bm{c}}\cdot\bm{x}=0\}$:
$$
  r_S = -\sum_{i=1}^n\frac{J_i}{\sqrt{c_i}}\cdot\sqrt{c_i}\na\mu_i
	= -\sum_{i,j=1}^n\frac{J_i}{\sqrt{c_i}}\cdot (P_L)_{ij}\sqrt{c_j}\na\mu_j
	= -\sum_{i=1}^n\frac{J_i}{\sqrt{c_i}}\cdot\sum_{j=1}^n(P_L)_{ij}\sqrt{c_j}\na\mu_j,
$$
where the projection $P_L$ on $L$ is defined in \eqref{2.PL}.
By the second law of thermodynamics, it should hold that $r_S\ge 0$.
To guarantee this property, we introduce
a positive semidefinite matrix $B(\bm{c})$ such that
\begin{equation}\label{app.B}
  \sum_{j=1}^n(P_L)_{ij}\sqrt{c_j}\na\mu_j
	= -\sum_{j=1}^n B_{ij}(\bm{c})\frac{J_j}{\sqrt{c_j}}, \quad i=1,\ldots,n.
\end{equation}
We claim that these equations correspond to the generalized Maxwell--Stefan 
equations \eqref{1.gms2}
after setting $J_i=c_iu_i$ and $K_{ij}(\bm{c})=\sqrt{c_i}B_{ij}(\bm{c})/\sqrt{c_j}$
(see \eqref{1.AK}). Indeed, the left-hand side of \eqref{app.B}, multiplied by
$\sqrt{c_i}$, becomes
$$
  \sqrt{c_i}\sum_{j=1}^n(P_L)_{ij}\sqrt{c_j}\na\mu_j
	= c_i\na\mu_i - c_i\sum_{j=1}^n c_j\na\mu_j,
$$
and the right-hand side of \eqref{app.B}, multiplied by $\sqrt{c_i}$, equals
$$
  -\sqrt{c_i}\sum_{j=1}^n B_{ij}(\bm{c})\frac{J_j}{\sqrt{c_j}}
	= -\sum_{j=1}^n K_{ij}(\bm{c})J_j = -\sum_{j=1}^n K_{ij}(\bm{c})c_ju_j.
$$
Hence, observing that $\mu_j=\pa U/\pa c_j$ corresponds to $\delta H/\delta c_j$,
\eqref{app.B} equals \eqref{1.gms2}.

\end{appendix}


\end{document}